\documentclass[12pt]{article}
\usepackage[english]{babel}
\usepackage{amsthm}
\usepackage{amsmath, color}
\usepackage{amsfonts}
\usepackage{amssymb}
\usepackage{enumerate}
\usepackage{float}
\usepackage{wrapfig}
\usepackage{graphicx}
\usepackage[thinc]{esdiff}
\usepackage[affil-it]{authblk}
\usepackage{fancyhdr}
\usepackage[thinc]{esdiff}
\usepackage{authblk}

\DeclareMathOperator{\Sign}{Sign}
\def\R{\mathbb{R}}

\def\vad{\vadjust{\kern-5pt}}
\def\a{\alpha}

    \def\F{\Phi}
 \def\ex{\exists\,}  
\def\dx{\dot x}     
   \def\p{\psi}
      
\def\le{\leqslant}  \def\ge {\geqslant}
\def\q{\quad}  \def\qq{\qquad}

\def\dis{\displaystyle} 
     \def\ms{\medskip}   \def\ssk{\smallskip}

\def\beq{\begin{equation}\label}  \def\eeq{\end{equation}}
\def\bth{\begin{theorem}\label}   \def\eth{\end{theorem}}
\def\ble{\begin{lemma}\label}     \def\ele{\end{lemma}}
\def\begar{\begin{array}} \def\endar{\end{array}}

\makeatletter
\renewcommand*{\@fnsymbol}[1]{\ensuremath{\ifcase#1\or *\or \ddagger\or
   \mathsection\or \mathparagraph\or \|\or **\or \dagger\dagger
   \or \ddagger\ddagger \else\@ctrerr\fi}}
\makeatother

\begin{document}
\title{Classification of extremals in a simplified Goddard model on the maximal height of rocket flight}

\author[1]{Andrei Dmitruk\thanks{dmitruk@member.ams.org}}
\author[2]{Ivan Samylovskiy\thanks{ivan.samylovskiy@cs.msu.ru}}

\affil[1]{Central Economics and Mathematics Institute of the 					  Russian Academy of Sciences}
\affil[2]{Lomonosov Moscow State University, Faculty of 						  Computational Mathematics and Cybernetics}

\renewcommand\Authands{ and }

\newtheorem{theorem}{Theorem}[section]
\newtheorem{prop}{Proposition}[section]
\newtheorem{rem}{Remark}[section]
\newtheorem{definition}{Definition}[section]
\newtheorem{Note}{Note}[section]
\newtheorem{lemma}{Lemma}[section]
\maketitle

\begin{abstract}
We consider a problem on maximizing the height of vertical flight of
a material point (``meteorological rocket'') in the presence of a nonlinear
friction and a constant flat gravity field under a bounded thrust and fuel
expenditure. The original Goddard problem is simplified by removing the
dependence on the rocket mass from the equations of motion. Using the
maximum principle we find all possible types of Pontryagin extremals and
classify them w.r.t. problem parameters. Since the velocity of the point
can be negative, we obtain some new types of extremals with two or three
switching points, which optimality should be further investigated.

\end{abstract}

\section{Introduction}

Consider the following optimal control problem:
\begin{equation}\label{probl}
    \left\{
    \begin{aligned}
        &\dot s= x,&&\quad s(0) = 0,&& \; s(T) \to \max, \\
        &\dot x= u-\varphi(x) - g,&&\quad x(0) = 0,&& \; x(t)\;\text{ is free},\\
        &\dot m= -u,&& \quad m(0)=m_0\,,&& \; m(T) \geqslant m_T\,,\\
        & 0\leqslant u \leqslant 1.
    \end{aligned}
    \right.
\end{equation}

Here $s(t)$ and $x(t)$ are one-dimensional position and velocity of a vehicle,
$m(t)$ describes the total mass of vehicle's body and fuel, $u(t)$ is the
rate of fuel expenditure, $g$ is a constant gravity force and $\varphi(x)$ is a function describing
the ``friction'' (media resistance) depending on the velocity. We assume that:
\begin{enumerate}[-]
\item $\varphi(0) = 0$,
\item $\varphi'(x)>0$ for all $x,$
\item $\varphi(x)$ is twice smooth for $x \neq 0$,
\item $\varphi''(x)<0$ for all $x<0$ and $\varphi''(x)>0$ for all $x>0$,
\end{enumerate}
which, in particular, implies that $\varphi(x)$ works on decreasing the
absolute value of the speed $|x|.$


This object can be considered as a material point moving vertically and
being forced by a nonnegative bounded thrust. Our aim is to maximize the
distance passed by the object in a given time $T$ under a fuel limitation.
Here $m_T\in ]0,m_0[$ is the mass of ``empty'' vehicle without fuel.
\ssk

This problem can be considered as a simplification of the classical
Goddard problem on maximizing the height of vertical flight of a
``meteorological rocket'' \cite{Goddard}, where the change of the objects's
mass is not taken into account in the equation for acceleration (i.e. $u$
is the thrust force divided by mass $m$).

The Goddard problem has a long history and was investigated both by the calculus of variations (see, for example, \cite{Leitmann1}, \cite{Leitmann2}) and the optimal control (see, for example, \cite{Leitmann3}, \cite{Bonnans1}, \cite{Martinon}, \cite{Petit2}) methods.

Our modification differs from the classical setting in the following aspects. On the one hand, as mentioned above, we do not take into account the reducing of mass in the equation for velocity, which simplifies the dynamics. On the other hand, we admit the presence of the gravity force together with a general nonlinear resistance and do not require the velocity to be nonnegative. To our knowledge, such a statement was not yet considered in the literature. It allows one to describe analytically all possible types of optimal trajectories, while in the original Goddard problem, in its full generality, it is hardly possible, and so, even qualitative properties of optimal trajectories are obtained by using numerical calculations (see, e.g. \cite{Bonnans1}, \cite{Martinon}).

\section{Preliminaries}
Note that since the admissible control set in problem (1) is convex and
compact, and the dynamics is linear in $u,$ the classical Filippov theorem
guarantees that a solution (an optimal trajectory) always exists.

To find an optimal trajectory, we, first of all, establish some properties of the control system
\begin{equation}\label{controlSystem}
\dot x = u - \varphi(x) - g, \qq 0 \leqslant u \leqslant 1.
\end{equation}

Define $x_{min} < 0$ and $x_{max} > 0$ from the conditions
\beq{x-min-max}
-\varphi(x_{min})-g=0, \qq 1-\varphi(x_{max}) - g = 0,
\eeq
respectively.\, Since $\varphi(x)$ strictly increases, these values
are unique.\, Moreover, the following proposition is true:

\begin{lemma}\label{xdot}\,
If $x_0\in\left]x_{min}, x_{max}\right[,$ then for any trajectory of
system (\ref{controlSystem}) we have $x(t)\in ]x_{min},x_{max}[$ for all
$t\ge 0.$

If $u=1$ ($u=0$) on a time interval $]t_1,t_2[,$ then $\dot x(t)> 0$
($<0,$ respectively) on this interval.


\end{lemma}

\begin{proof} \,
Set $u \equiv 1$ in the first equation of (\ref{controlSystem}) and note
that $x^{*}(t) \equiv x_{max}$ is its solution. Take any initial value
$x_1 < x_{max}$ and consider the solution $x'(t)$ of this equation for
$u \equiv 1$ with $x'(0) = x_0\,.$ Since the integral curves of the
same equation do not intersect, $x'(t) < x_{max}$ for all $t \ge 0\,.$

Now, consider any control $u(t)\in[0,1]$ and the corresponding $x(t)$
with the same $x(0) =x_0\,.$  According to Chyaplygin comparison theorem,
we obtain $x(t) \le x'(t) < x_{max}$ for $t \ge 0\,.$
Similarly, if $u \equiv 0,$ the corresponding solution $x''(t)$ satisfies
the relations $x(t) \ge x''(t) > x_{min}$ for $t \ge 0\,.$

These inequalities, in view of (\ref{x-min-max}), imply that $\dx(t) >0$
if $u=1,$ and  $\dx(t) <0$ if $u=0.$ 
\end{proof}

\section{Maximum principle for problem (\ref{probl})}

Let $s(t),\, x(t),\, m(t),\, u(t),\; t\in [0,T]$ be an optimal process.
According to the Pontryagin Maximum Principle (MP), there exist constants $\left(\alpha_0,\alpha,\beta_s,\beta_x,\beta_m\right),$ not all identically zero,
 and Lipschitz functions $\psi_s(t),\, \psi_x(t), \psi_m(t),$ that generate
{\it the endpoint Lagrange function}
\begin{equation}\label{l}
l=-\alpha_0 s(T)-\alpha (m(T)-m_T)+\beta_s s(0)+\beta_x x(0)+\beta_m (m(0)-m_0)
\end{equation}
and {\it the Pontryagin function}
\begin{equation} \label{H}
H(s,x,m,u)=\; \psi_s\, x + \psi_x\,(u - \varphi(x) - g)-\;\psi_m\,u,
\end{equation}
such that the following conditions are satisfied:

\begin{enumerate}[(a)]

\item nonnegativity condition: $\;\alpha_0\geqslant0,\;\;\alpha\geqslant 0,$
\ssk

\item nontriviality condition: $\;(\alpha_0,\alpha,\beta_s,\beta_x,\beta_m)\neq (0,0,0,0,0)$,
\ssk

\item
complementarity slackness condition:  \vad
\begin{equation}\label{KKT}
\alpha\, (m(T)-m_T)=0,
\end{equation}

\item costate (adjoint) equations
\begin{equation} \label{adjointSystem}
\begin{cases}
\begin{aligned}
-\dot\psi_s &= H_s =\; 0, \\
-\dot\psi_x &= H_x =\; \psi_s-\psi_x\varphi'(x),\\
-\dot\psi_m &= H_m =\; 0,
\end{aligned}
\end{cases}
\end{equation}

\item transversality conditions:
\begin{equation}\label{transvers}
\begin{cases}
\begin{aligned}
\psi_s(0)&=\beta_s\,,&&\q \psi_s(T)=\alpha_0\,,\\
\psi_x(0)&=\beta_x\,,&&\q \psi_x(T)=0\,,\\
\psi_m(0)&=\beta_m\,,&&\q \psi_m(T)=\alpha\,,
\end{aligned}
\end{cases}
\end{equation} \ssk

\item
the ``energy conservation law'': $\; H(s,x,m,u) \equiv const,$
\ssk

\item
and the maximality condition: \quad for almost all $t$  \vad
\begin{equation} \label{max}
\max_{0\le u'\le 1}\, H(s(t),x(t),m(t),u') = H(s(t),x(t),m(t),u(t)).
\end{equation}
\end{enumerate}  \ssk

\noindent
According to (\ref{adjointSystem})--(\ref{transvers}), in order to simplify
further computations we set $\psi_s\equiv \alpha_0,$ $\psi_m\equiv \alpha,$
and write $\psi(t)$ instead of $\psi_x(t).$ Then the maximality condition
(\ref{max}) gives the optimal control in the form  \vad
\begin{equation}\label{uOpt}
u(t)\in \Sign^+(\psi-\alpha),
\end{equation}
where $\Sign^+$ is the set-valued function
\begin{equation*}
\Sign^+(z)=
\begin{cases}
\begin{aligned}
&\left\lbrace 1\right\rbrace,&z>0,\\
&[0,1],&z=0,\\
&\left\lbrace 0\right\rbrace,&z<0,
\end{aligned}
\end{cases}
\end{equation*}
and the costate $\psi(t)$ is determined by the equation  \vad
\begin{equation} \label{adjointEq}
\dot\psi =\; -\alpha_0+\psi\,\varphi'(x) \qquad
\end{equation}
with the terminal condition $\,\psi(T) =0.$
\ssk

By the assumptions,  $\Delta m:=m_0-m_T>0.$ If $\Delta m\ge T,$ then
the optimal control is obvious: $u\equiv 1$ (full thrust).\,
So, in further considerations we assume that \vad
\begin{equation}\label{FeasibleM}
0< \Delta m < T.
\end{equation}

\section{Analysis of the maximum principle}

First, we show that the abnormal case $\alpha_0 = 0$ is impossible.
Suppose $\alpha_0=0.$ Then equation (\ref{adjointSystem}) for $\psi(t)$
restricts to a homogeneous one, and condition $\psi(T)=0$ yields $\psi(t)\equiv 0.$
Hence $\beta_x=0$ by (\ref{transvers}), and nontriviality condition gives
$\alpha > 0.$ Then (\ref{uOpt}) yields \mbox{$u(t)\equiv 0,$}
and from equations (\ref{probl}) we have $m(t)=const=m_0,$ which contradicts
complementarity slackness condition (\ref{KKT}). Hence, $\alpha_0>0$
and we may take $\alpha_0=1.\;$ Thus, equation (\ref{adjointEq}) reads
\begin{equation}\label{adjointEq1}
\dot\psi =\; -1+\psi\,\varphi'(x).
\end{equation}

Note that $\dot \psi (t)$ is a continuous function.

\begin{prop}\label{psigreaterzero}
$\; \psi(t)>0$ for all $t<T.$
\end{prop}

\begin{proof}\,
According to (\ref{adjointEq1}), $\dot \psi(T)= -1,$ and we know that
$\p(T) =0.$ Then $\psi(t)>0$ in a left neighborhood of $T.$ Suppose there
exists $t'<T$ such that $\psi(t')=0$ and $\psi(t)>0$ on $(t',T).$
From (\ref{adjointEq1}) we again have $\dot \psi(t')=-1,$ which contradicts
the previous inequality. 
\end{proof}

\begin{prop} \label{alpha}
$\;\; \alpha>0.$
\end{prop}

\begin{proof} \,
Suppose that $\alpha=0.$ Then from (\ref{uOpt}) and proposition (\ref{psigreaterzero}),
we obtain $u\equiv 1$ for a.a. $t.$ Hence $\Delta m=T,$ which contradicts
(\ref{FeasibleM}). 
\end{proof}

From this proposition it follows that $\ex t_2 \in (0,T)$ such that
$\psi(t) < \alpha$ (and then $u=0$ by (\ref{uOpt}))\, for all $t>t_2\,.$
Moreover, since $\alpha>0,$ condition (\ref{KKT}) gives $m(T)=m_T\,,$ and
hence  \vad
\begin{equation}\label{uIneq}
\int_0^T u\, dt \; =\; \Delta m\; >0.
\end{equation}

\begin{definition}
To define conveniently the control function, we use the notation
$$u = (u_1,\, u_2,\, \ldots)\;\; \mbox{ on }\;\; (\Delta_1,\, \Delta_2,\, \ldots),
$$
where $\Delta_1,$ $\Delta_2\,,$ etc., are some intervals, if $u(t) = u_1$
on $\Delta_1,\; $ $u(t) = u_2$ on $\Delta_2\,,$ etc.
\end{definition}

\begin{rem}
\normalfont
If the friction function is linear: $\varphi(x)=\gamma x$ $(\gamma>0),$ the analysis
of MP is quite simple. Then (\ref{adjointEq1}) determines
$\psi(t)=\left(1-e^{\gamma(t-T)}\right)/\gamma$ which is positive on $[0,T]$
and decreases monotonically from $\psi(0)>0$ to $\psi(T)=0$. In this case
condition (\ref{uOpt}) implies that the optimal control always has a bang-bang
form $u=(1,0)$ on $\left(]0,\Delta m[,\,]\Delta m,T[\right).$ Such a case is not interesting, and this is why we assume that $\varphi(x)$ is strictly convex for $x>0$ and strictly concave for $x<0.$
\end{rem}

Now, define the set $M=\left\{ t:\; \psi(t)=\alpha \right\}.$
Obviously, $M$ is closed. Moreover, it is not empty (otherwise $\psi<\alpha$
on $(0,T),$ hence $u\equiv 0,$ which contradicts (\ref{uIneq})).

\section*{Some properties of $x(t)$ and $\psi(t)$ related to the set $M$}

\begin{lemma}\label{xProps}
The following facts take place: \ssk

1) Let some $t'<t''<T$ be such that $\psi(t')=\psi(t'')=\alpha$ and
$\psi(t)>\alpha$ on $]t',t''[.$ Then $x(t')<0$ and $x(t'')\le -x(t').$
\ssk

2) Let some $t'<t''$ be such that $\psi(t')=\psi(t'')=\alpha$ and
$\psi(t)<\alpha$ on $]t',t''[.$ Then either $x(t')<0$ or we have two relations:
$x(t')>0$ and $x(t'')\le -x(t').$
\end{lemma}


\begin{proof} \, Case 1)\, From the conditions it follows that
$\dot \psi(t')\ge 0,$ $\dot \psi (t'') \le 0,$ and from equation (\ref{adjointEq1})
we obtain $\varphi'\left(x\left(t''\right)\right) \le
\varphi'\left(x\left(t'\right)\right),$ hence $x(t'') > x(t'),$
since $u=1$ and $x$ increases on $(t',t'')$ by Lemma~\ref{xdot}.

Then from the properties of $\varphi'(x)$ it follows immediately that
$x(t')<0$ and $x(t'')\leqslant -x(t').$

Case 2) is studied analogically.  
\end{proof}

Using Lemma \ref{xdot}, the following properties of the adjoint variable
can be easily stated:

\begin{lemma}\label{psiProps}
The following behaviour of $\psi(t)$ is impossible for the optimal trajectory:

1) There exists $t'<T$ such that $x(t') < 0,$ $\psi(t')=\alpha,$
$\dot \psi(t') = 0$ and $\psi(t)>\alpha$ (see Fig. \ref{inopt1}a)
or $\psi(t)<\alpha$ (Fig. \ref{inopt1}b) in a right neighborhood
of $t'$.  \ssk

2) There exists $0<t''$ such that $x(t')\leqslant 0,$ $\psi(t'')=\alpha,$
$\dot \psi(t'') = 0$ and $\psi(t)>\alpha$ (see Fig. \ref{inopt2}a)
or $\psi(t)<\alpha$ (Fig. \ref{inopt2}b) in a left neighborhood of
$t''$.
\end{lemma}

\begin{figure}[h!]
\begin{minipage}[h]{0.35\linewidth}
\centering{$\includegraphics[width=\linewidth]{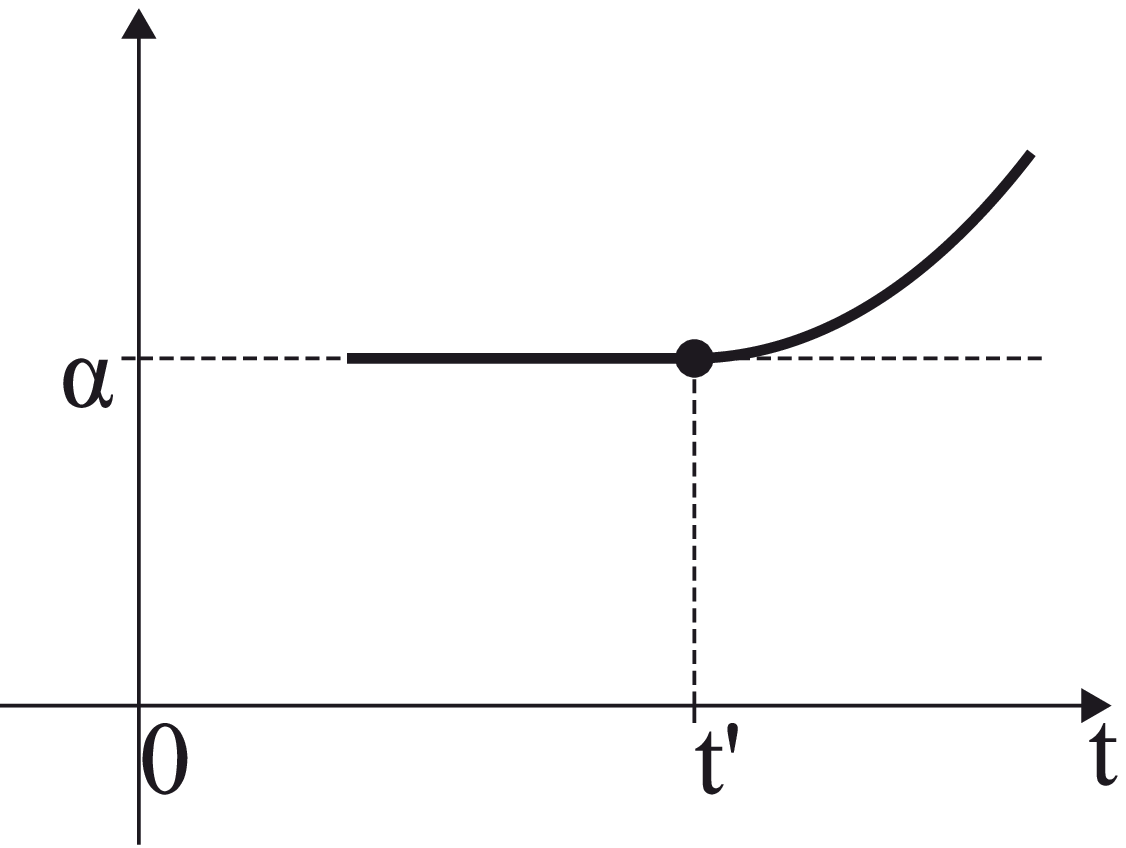}$}
\center{a: $x(t')< 0$}
\end{minipage}
\hfill
\begin{minipage}[h]{0.35\linewidth}
\centering{$\includegraphics[width=\linewidth]{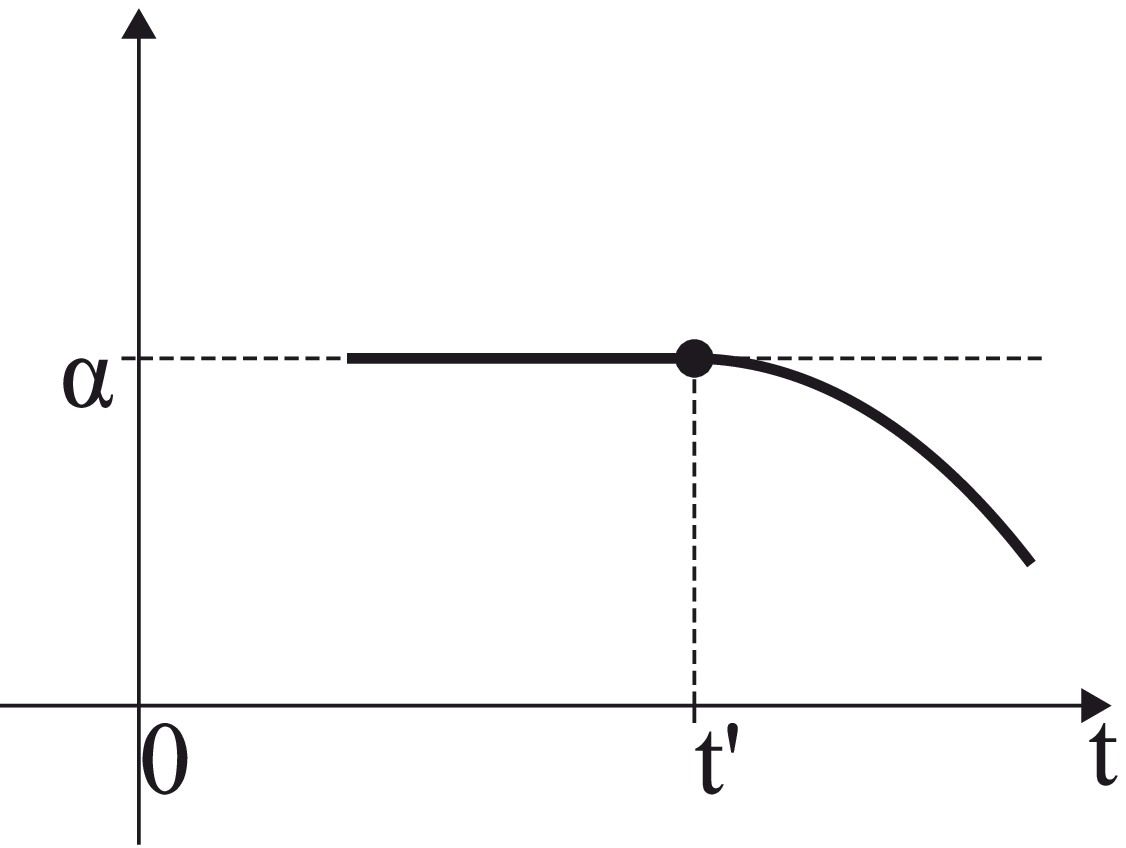}$}
\center{b: $x(t')\le 0$}
\end{minipage}
\caption{Starting from level $\varphi(t')=\alpha$ with $\dot \varphi(t')=0$}
\label{inopt1}
\end{figure}

\begin{figure}[h!]
\begin{minipage}[h]{0.35\linewidth}
\centering{$\includegraphics[width=\linewidth]{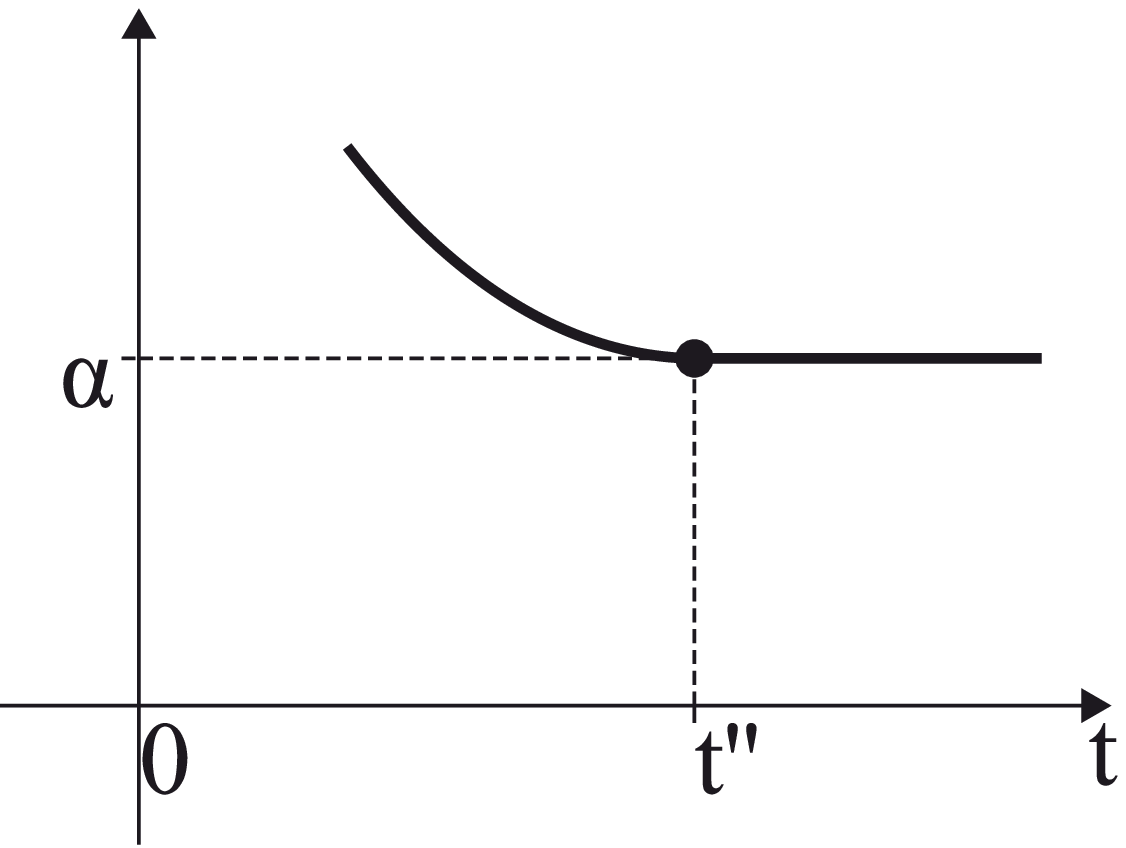}$}
\center{a: $x(t'')\le 0$}
\end{minipage}
\hfill
\begin{minipage}[h]{0.35\linewidth}
\centering{$\includegraphics[width=\linewidth]{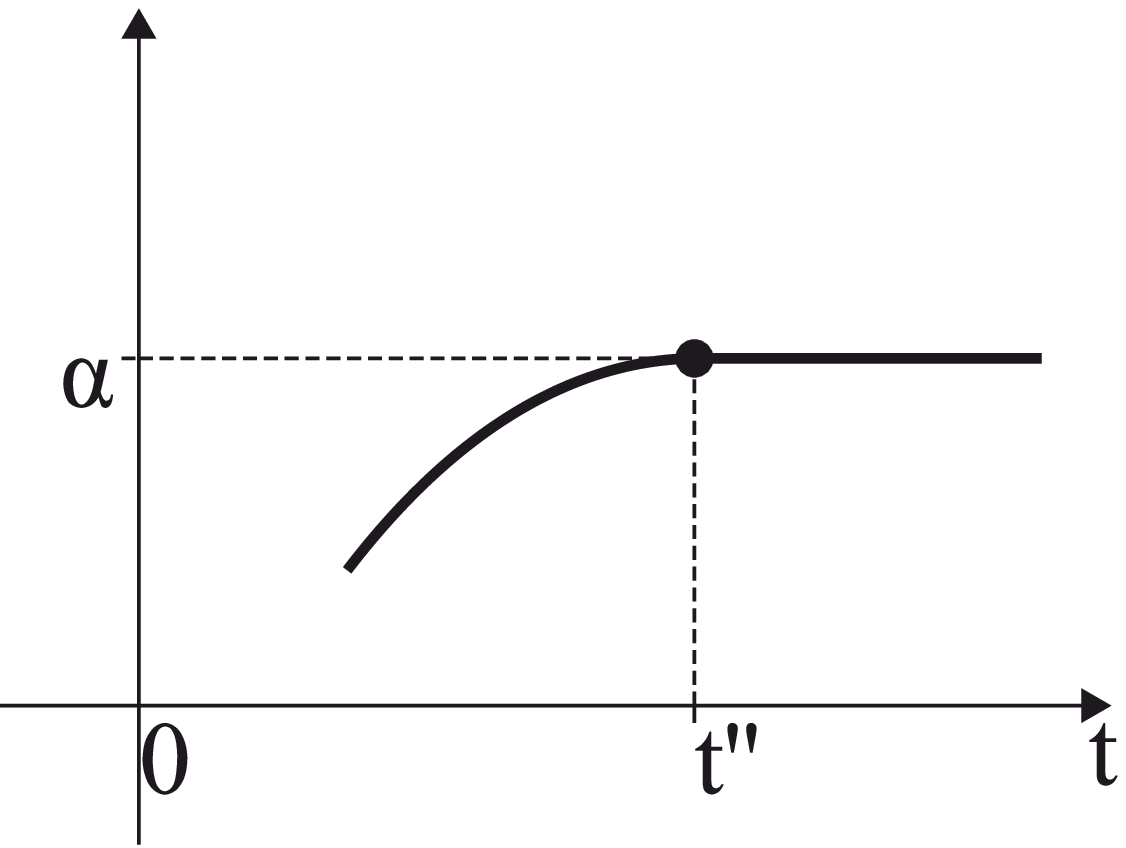}$}
\center{b: $x(t'')< 0$}
\end{minipage}
\caption{Reaching level $\psi(t'')=\alpha$ with $\dot \psi(t'')=0$}
\label{inopt2}\end{figure}


\begin{proof}\,
Using equation (\ref{adjointEq1}),  we can write the second time derivative
of $\psi$ as follows:
\begin{equation}\label{psiSecondDeriv}
\ddot{\psi}=\dot\psi \varphi^{\prime}(x) + \psi\varphi^{\prime\prime}(x)\dot x.
\end{equation}

Consider the case 1a, where $\psi > \alpha$ in a right neighborhood of $t'.$
Then $u = 1$ there, and Lemma \ref{xdot} gives $\dot x > 0.$
Since $x(t) < 0$ near $t',$ we have $\varphi''(x) < 0$ and then (\ref{psiSecondDeriv})
gives $\ddot{\psi} < 0$ in a right neighborhood of $t'.$
From the condition $\dot{\psi}(t') = 0$ we obtain $\psi(t) < \alpha$ in
a right neighborhood of $t',$ a contradiction.

The cases 1b (where $\psi(t) < \alpha$ in a right neighborhood of $t'$),
2a and 2b are analyzed similarly.  
\end{proof}

\section*{Crossing the level $\psi=\alpha$}

Let us find the values $x(t)$ corresponding to the time moments from $M.$

According to MP, $H=const$ along the optimal trajectory.
Thus,
$$x + \psi \left(u - \varphi(x) - g\right) - \alpha u = const$$
for  $0\le t \le T.$
For $t\in M,$ we have
$$
x + \alpha \left( u - \varphi(x) - g\right) - \alpha u = const,
$$
which yields
$$
x - \alpha \varphi(x) = const,
$$
or, finally, there exists such a constant $c$ that
$$
\varphi(x) - \frac{x}{\alpha} = c \;\quad \mbox{ for all }t\in M.
$$

Note that $\alpha > 0$ by Proposition \ref{alpha}.\, During the
following steps we will use this information to exclude inoptimal
trajectories from the analysis.

Our plan is as follows. We will find values which $x(t)$ can take on $M$
and then find what optimal trajectories correspond to them.
\ssk

For every constant $c$ define the sets
$$
X_C:=\left\lbrace x= x(t) \quad | \quad t \in M, \quad
\varphi(x) - \frac{x}{\alpha} = c\right\rbrace
$$
and
$$  M_C:=\left\lbrace t \in M \quad | \quad x(t) \in X_C\right\rbrace
$$
and use our knowledge on the properties of the set $M$ to narrow the
family of trajectories that satisfy MP.

\section*{Some properties of the set $X_C$}
\begin{lemma}\label{xMInopCase}
If $X_C=\left\lbrace x_1\right\rbrace$ is a singleton, then $M$ is a segment
or a point. If, in addition, $x_1 \le 0,$ then the corresponding trajectory
does not satisfy MP.
\end{lemma}

\begin{proof}\,
From Lemma \ref{xdot} it follows immediately that $M$ is either a point
or a segment. (If $\psi(t)$ intersects the level $\alpha$ twice, then
$X_C$ should contains two values). Let us show that $x_1 \le 0$ is not
possible. Indeed, if $M$ is a segment then we have a contradiction
with Lemma \ref{psiProps} (see Figure 2b).

If $M$ is a point and $x_1 < 0$ then the only possibility is that		
$\psi(t) < \alpha$ on $(0, t_1)$ and $\psi(t_1) = \alpha,$ where $t_1$
is a left bound of $M$. But since $\dot \psi$ is a continuous function,
this means that $\psi(t)>\alpha$ in a right neighborhood of $t_1$ and then
$\psi(t)$ should cross the level $\alpha$ once more to satisfy the
transversality condition $\psi(T)=0,$ a contradiction.

If $M$ is a point and $x_1 = 0,$ then obviously $\psi = \alpha$
at $t=0$ and decreases on $(0, T)$ to zero, thus $u\equiv 0,$
a contradiction with $\Delta m > 0.$ 
\end{proof}

\section{Classification of trajectories}

For any $\alpha > 0$ consider the function 
$\Phi(x) = \varphi(x) - \frac{x}{\alpha}\;$ and define 
$$
c_{min} = \min_{x\in [x_{min}, x_{max}]} \Phi(x), \qq
c_{max} = \max_{x\in [x_{min}, x_{max}]} \Phi(x).
$$
Then, for any $c\in [c_{min}, c_{max}],$ equation
\begin{equation}\label{equationC}
\varphi(x) - \frac{x}{\alpha} = c
\end{equation}
has a solution $x\in [x_{min}, x_{max}],$ i.e., the set $X_C$ has 
a nonempty intersection with $[x_{min}, x_{max}]$.  \ssk 

In addition, define $x^{\sim}$ from the relations  
\begin{equation*}
\varphi'(x^{\sim})= 1/\alpha\,, \qquad x^{\sim} > 0.
\end{equation*}
Since $\varphi'(x)$ strictly increases for $x > 0,$ $x^\sim$ is uniquely defined.
\ssk

Depending on $\alpha,$ the function $\Phi(x)$ can have one, two or
three roots on $[x_{min}, x_{max}],$ including the trivial root $x=0.$ 
In what follows we will use this fact to obtain different types of 
extremals.

For example, if for a given $\alpha$ the function $\Phi(x)$ has only zero
root on $[x_{min}, x_{max}],$ then $X_C$ consists of one value and, 
according to Lemma \ref{xMInopCase}, the only two possible types of extremals 
for this case are the well-known ``bang-bang'' or ``bang-singular-bang'' 
ones (see, e.g., \cite{DmiSam2}).  \ssk 

In view of this, let us introduce the following partition of the range of
parameters $g,\alpha,c.$

\begin{enumerate}[I.]
\item $0<g<0.5.$

Then $1-g>g,$ whence to $x_{min}>-x_{max}$ and $\frac{x_{max}}{1-g}<-\frac{x_{min}}{g}.$

\begin{enumerate}[1.]
\item $ \dis \alpha \le \frac{x_{max}}{1-g}\,.$

\begin{enumerate}[a.]
\item $x_{max}\leqslant x^{\sim}$.

\begin{figure}[H]\label{Ia1}
\centering{$\includegraphics[width=0.65\linewidth]{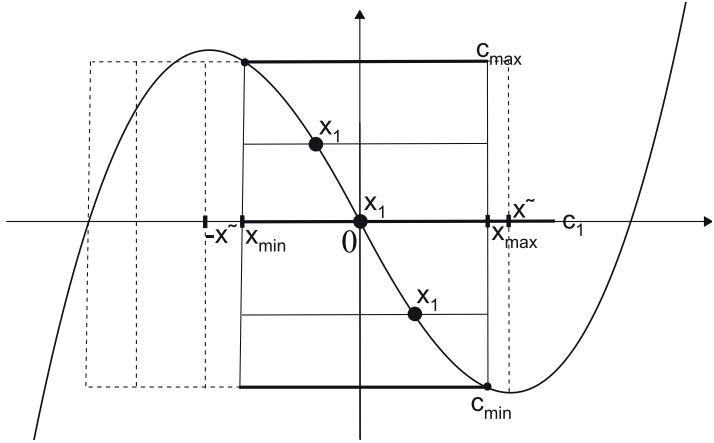}$}
\caption{Case I.1.a}
\end{figure}

Here for any $c\in[c_{min}, c_{max}]$ equation (\ref{equationC}) has 
a unique solution $x_1,$ i.e. $X_C=\left\lbrace x_1 \right\rbrace.$ 
Set $c_1 = \varphi(0).$ Thus, in this case the partition in $c$ is as follows:

\begin{enumerate}[1)]
\item $c_{min}\leqslant c < c_1.$ \ssk 

Here $x_1\in \left]0,x_{max}\right[.$ According to Lemma \ref{xMInopCase}, 
the set $M$ is either a point or a segment. Hence, the following two types 
of behavior of $\psi(t)$ and corresponding $x(t)$ are possible.
\ssk

\q The first one is when $\psi(t)$ decreases on $[0, t_1]$ from 
$\psi(0)>\alpha$ to $\alpha$, where $x(t_1)=x_1,$ crosses the level 
$\psi(t_1)=\alpha$ at $t_1$ and then decreases on $(t_1,T]$ from $\alpha$ 
to $0.$ Thus, $u=(1,0)$ on the intervals $]0,t_1[,$ $]t_1,T[,$ i.e. has 
the bang-bang form. Attribute such a trajectory to the \emph{type Ia}. 
\ssk

\q The second one is when $\psi(t)$ comes to the level $\psi(t_1)=\alpha$ 
with zero time derivative, stays at the level $\alpha$ on a time interval 
$]t_1,t_2[,$ and then decreases on $]t_2,T[$ from $\alpha$ to $0.$

Here the control is bang-singular-bang, and, as usual, to find the value 
of singular control on $]t_1,t_2[,$ we differentiate the equality 
$\psi(t)\equiv \alpha,$ obtaining $\dot \psi= -1+\alpha\varphi'(x)\equiv 0.$
Since the function $\varphi'(x)$ is strictly monotone, $x(t)= const,$ 
whence $\dot x= u-\varphi(x) - g=0,$ and so, \vad 
\begin{equation}\label{uSng}
u_{sing}(t)=\varphi(x(t)) + g \qquad \mbox{on}\;\;\; (t_1,t_2).
\end{equation}
Thus, here we get $u=(1,u_{sing},0)$ on
$(]0,t_1[,\,]t_1,t_2[,\,]t_2,T[)$ with a singular subarc $]t_1,t_2[.$

Note that, for a given starting point $t_1$ of singular subarc, the corresponding
endpoint is uniquely determined as
\begin{equation}\label{t2formula}
t_2\;=\; t_1+\frac{m_0-t_1-m_T}{\varphi\left(x(t_1)\right) + g},
\end{equation}
which can be easily seen from equation
$$
m_0 - t_1 + u_{sing}(t_2 - t_1) =\, m_T\,.
$$
Attribute such a trajectory to the \emph{type Ib}. \ssk 

\item $c = c_1$.

Here $x_1 = 0$ and Lemma \ref{xMInopCase} does not give us extremals.
\ssk

\item $c_{1} < c \leqslant c_{max}\,.$

Here $x_1\in \left]x_{min}, 0\right]$ and by Lemma \ref{xMInopCase}, 
such a trajectory does not satisfy the maximum principle.

\end{enumerate}

\item $x_{max}>x^{\sim},$  $\; x_{min}\geqslant -x^{\sim}\,.$

\begin{figure}[h]\label{Ia2}
\centering{$\includegraphics[width=0.65\linewidth]{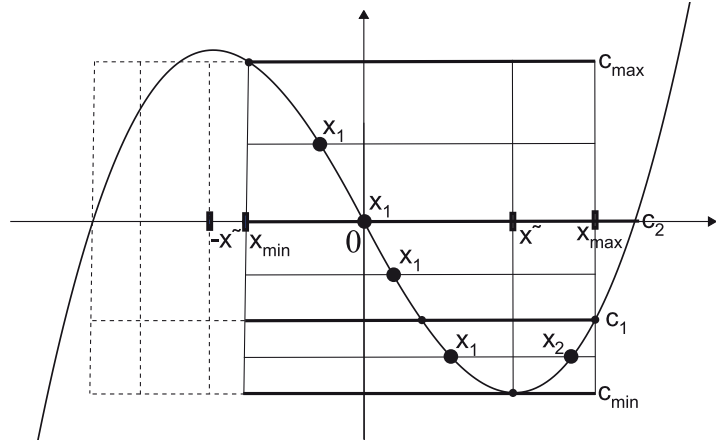}$}
\caption{Case I.1.b}
\end{figure}

Here we define $c_1 = \varphi(x_{max}) - \frac{x_{max}}{\alpha},$ 
$\q c_2 = \varphi(0).$

In this case the partition w.r.t. $c$ is as follows:

\begin{enumerate}[1)]
\item $c=c_{min}\,.$

Here $X_C=\left\lbrace x^\sim\right\rbrace,$ which gives us either a trajectory of type Ia or a trajectory of type Ib.
\ssk

\item $c_{min}<c\leqslant c_1\,.$

Here $X_C=\left\lbrace x_1, x_2\right\rbrace,$ where $x_1 \in ]0,x^\sim[$ and $x_2 \in ]x^\sim, x_{max}[$ which gives us a trajectories of type Ia or Ib (with $x(t_1)=x_1$ or $x(t_1)=x_2$).
\ssk

\item $c_1<c<c_2\,.$

Here $X_C=\left\lbrace x_1 \right\rbrace,$ where $x_1 \in]0,x^\sim[,$ which gives us a trajectory of type Ia or Ib.
\ssk

\item $c_2 \leqslant c\leqslant c_{max}\,.$

Here $X_C=\left\lbrace x_1 \right\rbrace,$ where $x_1 \in ]x_{min},0[$ 
and by Lemma \ref{xMInopCase} this case does not give us extremals. 
\end{enumerate}

\item $x_{max}>x^{\sim},$ $\; x_{min}< -x^{\sim}\,.$

\begin{figure}[h!]\label{Ia3}
\centering{$\includegraphics[width=0.65\linewidth]{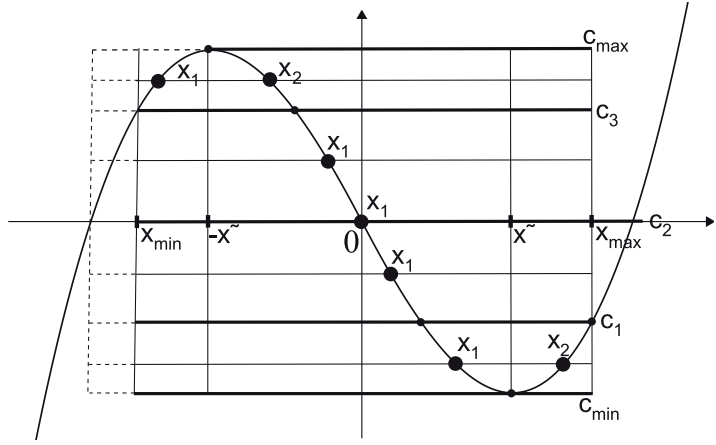}$}
\caption{Case I.1.c}
\end{figure}
\end{enumerate}

Define $c_1 = \varphi(x_{max}) - \frac{x_{max}}{\alpha},\;$ 
$c_2 = \varphi(0),\;$ $c_3 = \varphi(x_{min}) - \frac{x_{min}}{\alpha}\,.$
\ssk

In this case the partition in $c$ is as follows:

\begin{enumerate}[1)]

\item $c=c_{min}\,.$

Here $X_C=\left\lbrace x^\sim\right\rbrace,$ which gives us either a trajectory 
of type Ia or a trajectory of type Ib.
\ssk

\item $c_{min}<c\le c_1\,.$

Here $X_C=\left\lbrace x_1, x_2 \right\rbrace,$ where $x_1 \in ]0,x^\sim[$ 
and $x_2 \in]x^\sim, x_{max}[,$ which gives us a trajectory of type Ia 
or Ib (with $x(t_1)=x_1$ or $x(t_2)=x_2$).
\ssk

\item $c_1< c <c_2\,.$

Here $X_C=\left\lbrace x_1 \right\rbrace,$ where $x_1 \in]0,x^\sim[,$ 
which gives us a trajectory of type Ia or Ib.
\ssk 

\item $c_2 \leqslant c< c_3\,.$

Here $X_C=\left\lbrace x_1 \right\rbrace,$ where $x_1 \in ]x_{min},0[$ 
and by Lemma \ref{xMInopCase} this case does not give us extremals. 
\ssk 

\item $c_3\leqslant c < c_{max}\,.$

Here $X_C=\left\lbrace x_1, x_2\right\rbrace,$ where $x_1 \in]x_{min},-x^\sim[$ 
and $x_2 \in]-x^\sim,0[ ,$ and according to Lemma \ref{xProps} and 
Lemma \ref{psiProps} there is only one possible variant of the 
behavior of $\psi(t)$ and corresponding $x(t)$.

Here, $\psi(t)$ increases on $]0, t_1[$ from $\psi(0)<\alpha$ to $\alpha,$ 
then crosses the level $\alpha$ with a strictly positive time derivative, 
returns to value $\alpha$ at a time $t_2>t_1$ and then crosses the value 
$\alpha$ with negative time derivative and decreases on $]t_2,T[$ from 
$\alpha$ to $0.$

Thus, $u=(0,1,0)$ on the intervals $]0,t_1[$, $]t_1,t_2[$, $]t_2,T[,$ 
which gives us a \emph{trajectory of type IIa}.
\ssk

\item $c=c_{max}\,.$

Here $X_C=\left\lbrace -x^\sim\right\rbrace$ and according to Lemma \ref{xMInopCase} 
this case does not give us trajectories satisfying the maximum principle.
\end{enumerate}

\item $\dis \frac{x_{max}}{1-g} \leqslant \alpha< -\frac{x_{min}}{g}\,.$

\begin{enumerate}[a.]

\item $x_{min}\geqslant -x^{\sim}.$
\begin{figure}[h]\label{Ib1}
\centering{$\includegraphics[trim = 0 50 0 15 , clip,width=\linewidth]{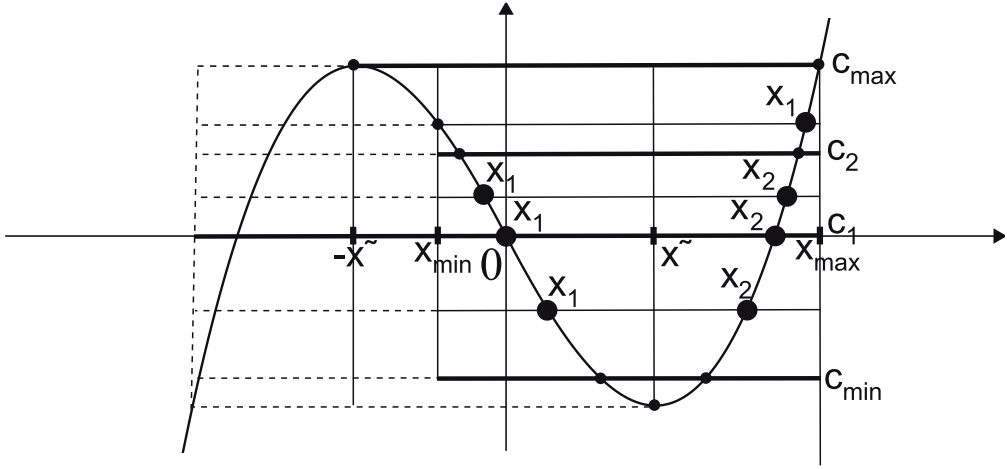}$}
\caption{Case I.2.a}
\end{figure}

Define $c_1 = \varphi(x_{0}),$ 
$\; c_2 = \varphi(x_{min}) - \frac{x_{min}}{\alpha}\,.$

Thus, in this case the partition in $c$ is as follows:

\begin{enumerate}[1)]

\item $c=c_{min}\,.$

Here $X_C=\left\lbrace x^\sim\right\rbrace,$ which gives us either a 
trajectory of type Ia or a trajectory of type Ib.
\ssk

\item $c_{min}<c< c_1\,.$

Here $X_C=\left\lbrace x_1, x_2 \right\rbrace,$ where $x_1 \in]0,x^\sim[$ 
and $x_2 \in]x^\sim, x_{max}[ ,$ which gives us again a trajectory of 
type Ia or Ib (with $x(t_1)=x_1$ or $x(t_1)=x_2$).
\ssk 

\item $c_1 \leqslant c < c_{max}\,.$

Here $X_C=\left\lbrace x_1, x_2\right\rbrace,$ where $x_1 \in]x_{min},0[$ 
and $x_2 \in ]x^\sim, x_{max}[,$ which gives us only trajectories of 
type Ia or Ib.  \ssk 

\item $c=c_{max}\,.$

Here $X_C=\left\lbrace x_1=-x^\sim, x_2=x_{max}\right\rbrace.$
According to Lemma \ref{xdot}, $X_C=\left\lbrace -x^\sim \right\rbrace,$ 
which contradicts Lemma \ref{xMInopCase}.

\end{enumerate}

\begin{rem}\,
Case 
$\F(x_{max})  > \F( - x^\sim) $ 
does not add new trajectories to the final picture (since $X_C = {x_1}$).
\end{rem}

\item $x_{min}< -x^{\sim}$
\begin{figure}[h]\label{Ib2}
\centering{$\includegraphics[trim = 0 50 0 30, clip, width=0.9\linewidth]{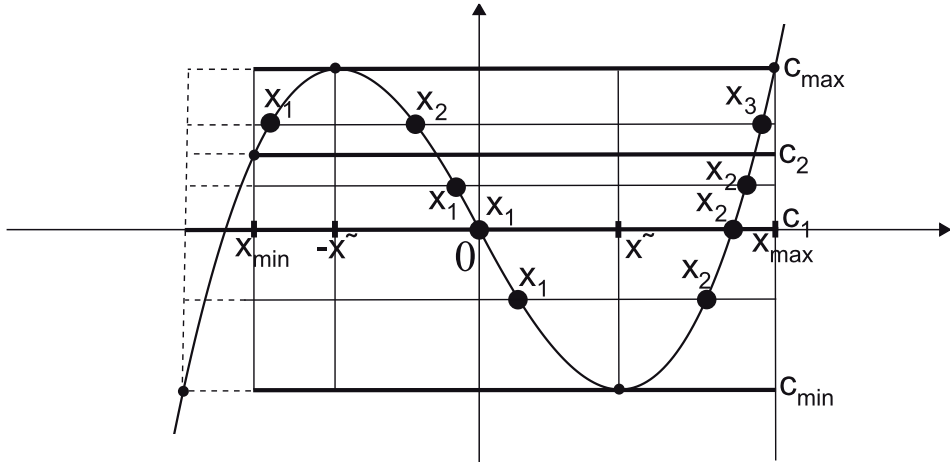}$}
\caption{Case I.2.b}
\end{figure}

Define $c_1 = \varphi(0),\;$ $c_2 = \varphi(x_{min}) - \frac{x_{min}}{\alpha}\,.$

In this case the partition in $c$ is as follows:

\begin{enumerate}[1)]
\item $c=c_{min}\,.$

Here $X_C=\left\lbrace x^\sim\right\rbrace,$ which gives us either a trajectory 
of type Ia or a trajectory of type Ib.  \ssk 

\item $c_{min}<c < c_1\,.$

Here $X_C=\left\lbrace x_1, x_2 \right\rbrace,$ where $x_1 \in]0,x^\sim[$ 
and $x_2 \in]x^\sim, x_{max}[,$ hence we obtain a trajectory of type Ia 
or Ib (with $x(t_1)=x_1$ or $x(t_2)=x_2$).
\ssk

\item $c_1 \leqslant c< c_2\,.$

Here $X_C=\left\lbrace x_1, x_2,\right\rbrace$ where $x_1 \in ]-x^\sim, 0[$ 
and $x_2 \in ]x^\sim, x_{max}[,$ and  due to Lemma \ref{xMInopCase} 
we obtain only extremals of type Ia or Ib.   \ssk 

\item $c_2\leqslant c< c_{max}\,.$

Here $X_C=\left\lbrace x_1, x_2, x_3\right\rbrace,$ where $x_1 \in]x_{min}, -x^\sim[,$ 
$x_2 \in]-x^\sim,0[$ and $x_3 \in]x^\sim,x_{max}[,$ and due to the 
Lemma \ref{xProps}, Lemma \ref{psiProps} and Lemma \ref{xMInopCase} we 
have trajectories of types Ia, Ib, IIa. Moreover, $\psi(t)$ can have
the following behavior:\, it increases on $]0, t_1[$ from $\psi(0)<\a$ 
to $\alpha,$ then crosses the level $\alpha$ with a strictly positive 
derivative, returns to value $\alpha$ at time $t_2>t_1$ with zero derivative, 
then stays on the level $\alpha$ on an interval $]t_2,t_3[$ and then 
decreases from $\alpha$ to $0$ on $]t_3,T[$.

Thus, $u=(0,1,u_{sing},0)$ on the intervals $]0,t_1[$, $]t_1,t_2[$, 
$]t_2,t_3[,$ $]t_3,T[,$ which gives us a \emph{trajectory of type IIb}.
\ssk

\item $c=c_{max}\,.$

Here $X_C= \{ x_1= -x^\sim,\; x_2=x_{max}\}$ and due to Lemma \ref{xMInopCase} 
we will not obtain any extremals. 

\end{enumerate}

\begin{rem}\,
Case 
$\F(x_{max})  > \F( - x^\sim) $ 
does not add new trajectories to the final picture (since $X_C = {x_1}$).
\end{rem}

\end{enumerate}

\item $\alpha\geqslant -\frac{x_{min}}{g}\,.$

\begin{figure}[h]
\center{$\includegraphics[trim = 0 50 0 30, clip, width=0.9\linewidth]{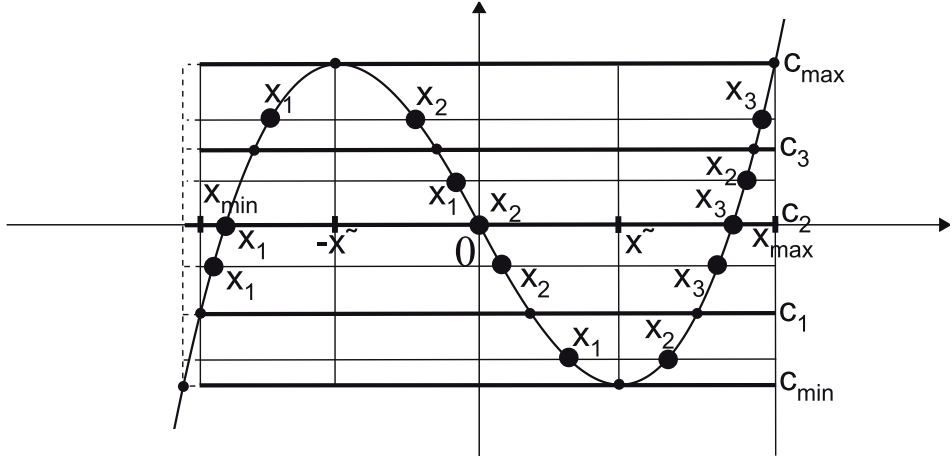}$}
\caption{Case I.3}
\end{figure}

Define $c_1 = \varphi(x_{min}) - \frac{x_{min}}{\alpha},\;$ 
$c_2 = \varphi(0).$

Thus, in this case the partition in $c$ is as follows.

\begin{enumerate}[1)]
\item $c=c_{min}\,.$

Here $X_C=\left\lbrace x^\sim\right\rbrace,$ which gives us either a 
trajectory of type Ia or a trajectory of type Ib.  \ssk

\item $c_{min}< c < c_1\,.$

Here $X_C=\left\lbrace x_1\, x_2\right\rbrace,$ where $x_1 \in]0,x^\sim[$ 
and $x_2 \in]x^\sim, x_{max}[,$ which gives us a trajectory of type Ia 
or Ib (with $x(t_1)=x_1$ or $x(t_1)=x_2$). \ssk 

\item $c_1\leqslant c< c_2\,.$

Here $X_C=\left\lbrace x_1, x_2, x_3 \right\rbrace,$ where 
$x_1 \in]x_{min},-x^\sim[,$ $x_2 \in ]0, x^\sim[$ and $x_3\in]x^\sim,x_{max}[,$ 
which gives us either trajectories of types Ia, Ib, IIa, IIb, or the 
following behavior of $\psi(t).$

Obviously the only variant for $\psi(t)$ to cross the level $\alpha$ at 
three time moments corresponding to three values from $X_C$ is first to 
decrease on $[0,t_1]$ from $\psi(0)>\alpha$ to $\alpha,$ then cross down 
the level $\alpha,$ return to level $\alpha$ ``from below'' on $[t_1, t_2],$ 
then cross up the level $\alpha,$ return to level $\alpha$ ``from above'' 
on $[t_2, t_3],$ and finally decrease from $\alpha$ to $0$ on $[t_3,T].$ 
Of course, it is possible if and only if $x(t_1) = x_3$ and $x(t_3) = x_2\,,$ 
otherwise we obtain $\dot \psi(t_3)> \dot \psi(t_1)$ and due to the 
comparison theorem we have a contradiction with the transversality condition 
$\psi(T) = 0.$ More precisely, in such a case we would obtain 
$\psi(t_1)=\psi(t_3),\,$ $\dot \psi(t_1) < \dot \psi(t_3),$ according to 
(\ref{adjointEq1}), and $\varphi' ( x(t_1)) < \varphi (x(t_3))$ since 
$0<x_2<x_3\,.$ Then, the comparison theorem implies 
$\psi(t)>\psi\left( t - (t_3 - t_1)\right)$ for all $t > t_3\,,$ which 
contradicts the transversality condition.
\ms 

Let us call trajectories corresponding to the ``basic'' case, in which 
$u = (1, 0, 1, 0)$ on the consecutive time intervals $]0, t_1[,$ $]t_1, t_2[,$ 
$]t_2, t_3[,$ $]t_3, T[,$ \emph{trajectories of type III}.
\ssk

Note that Lemma \ref{psiProps} seemingly allows us to expand either $t_1$ or $t_3$ or both of them to a singular subarc, but it can be easily shown that such a expansion leads to a contradiction.

Indeed, consider expansion of a single point $t_1$ to a singular subarc. 
Then we have $\dot \psi(t_1) = 0,$ i.e. $-1 + \alpha \varphi'(x_3) = 0,$
which is equal to
$$
\left.\left(\varphi(x) - \frac{x}{\alpha}\right)'\right|_{x=x_3} = 0.
$$
Since $x_3>0,$ we have $x_3 = x^\sim$ with account of properties of 
$\varphi(x) - \frac{x}{\alpha}.$ Thus, $X_C$ contains only two points, 
a contradiction.  \ssk 

\item $c= c_2\,.$

Here $X_C=\{ x_1,\, x_2=0,\, x_3=-x_1 \},$ where $x_1\in]x_{min},-x^\sim[,$ 
which gives us either a trajectory of type 0\, or a trajectories of types 
Ib, IIb.  \ssk 

\item $c_2 <  c<c_{max}\,.$

Here $X_C=\left\lbrace x_1,\, x_2,\, x_3\right\rbrace,$ where 
$x_1\in]x_{min}, -x^\sim[,$ $x_2\in]-x^\sim,0[$ and 
$x_3\in]x^\sim,x_{max}[,$ and due to the Lemmas \ref{xProps}, 
\ref{psiProps}, and \ref{xMInopCase} we have trajectories of types 
Ia, Ib, IIa, IIb.  \ssk 

\item $c=c_{max}\,.$

Here $X_C=\left\lbrace x_1= -x^\sim, x_2=x_{max}\right\rbrace.$ 
Lemma \ref{xdot} excludes $x_2 = x_{max}$ from $X_C\,,$ and since $x_1<0,$ 
Lemma \ref{xMInopCase} does not give us trajectories satisfying MP.
\end{enumerate}

\begin{rem}\,
Cases 
$\F(x_{max}) > \F( - x^\sim)$ 
and 
$\F(x_{min})  < \F(x^\sim) $ 
do not add new trajectories to the final picture (since $X_C = {x_1}$).
\end{rem}

\end{enumerate}

\item $g=0.5.$

Here $x_{min}= -x_{max},$ and partition from the previous section reduces 
to the following one:

\begin{enumerate}[1)]
\item $\alpha \leqslant 2 x_{max}$

\begin{enumerate}[a.]

\item $x_{max} \leqslant x^\sim$

The situation is the same as in the case I.1.a

\item $x_{max} > x^\sim$

The situation is the same as in the case I.1.c

\end{enumerate}

\item $\alpha > 2 x_{max}$

The situation is the same as in the case I.3
\end{enumerate}
\ms

\item $0.5<g<1$

Then $\dis x_{min}<-x_{max}$ and $\frac{x_{max}}{1 - g} > - \frac{x_{min}}{g}.$

\begin{enumerate}[1.]
\item $\alpha \leqslant - \frac{x_{min}}{g}$

\begin{enumerate}[a.]

\item $x_{min} \geqslant - x^\sim.$

The situation is the same as in the case I.1.a.

\item $x_{min} < - x^\sim, \quad x_{max} \leqslant x^\sim$

\begin{figure}[h]
\centering{$\includegraphics[trim=0 50 0 50, clip, width=0.8\linewidth]{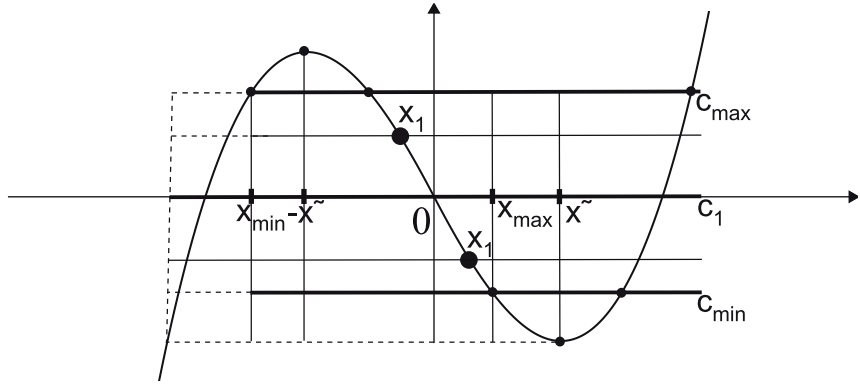}$}
\caption{Case III.1.b}
\end{figure}

Define $c_1 = \varphi(0).$
Thus, in this case the partition in $c$ is as follows:

\begin{enumerate}[1)]

\item $c_{min} \leqslant c < c_1$

Here $X_C=\left\lbrace x_1 \right\rbrace,$ where $x_1\in ]0,x^\sim]$, 
which gives us either a trajectory of type Ia or a trajectory of type 
Ib. \ssk 

\item $c = c_{1}.$

Here $x_1 = 0,$ which gives us no extremals. \ssk 

\item $c_{1} < c \leqslant c_{max}.$

Here $x_1\in \left]x_{min}, 0\right].$ By Lemma \ref{xMInopCase}, 
the corresponding trajectories do not satisfy MP. 
\end{enumerate}

\item $x_{max} > x^\sim$

\begin{figure}[h]
\center{$\includegraphics[trim=0 70 0 30, clip, width=0.8\linewidth]{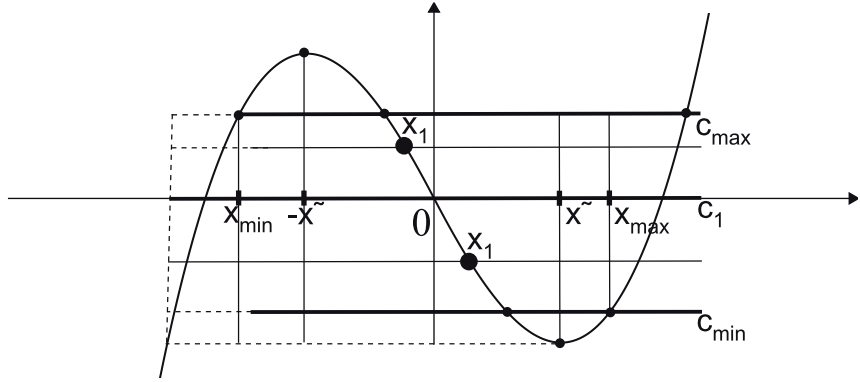}$}
\caption{Case III.1.c}
\end{figure}

Define $c_1 = \varphi(0).$
Thus, in this case the partition in $c$ is as follows:

\begin{enumerate}[1)]

\item $c_{min} \leqslant c < c_1$

Here $X_C=\left\lbrace x_1, x_2 \right\rbrace,$ where $x_1\in ]0,x^\sim[$ 
and $x_2 \in ]x^\sim, x_{max}[$, which gives us either a trajectory of 
type Ia or a trajectory of type Ib. \ssk 

\item $c = c_{1}.$

Here $x_1 = 0,$ which gives no trajectories satisfying MP. \ssk 

\item $c_{1} < c \leqslant c_{max}.$

Here $x_1\in \left]x_{min}, 0\right].$ By Lemma \ref{xMInopCase},
the corresponding trajectories do not satisfy MP. 
\end{enumerate}

\end{enumerate}
\ms

\item 
$\dis - \frac{x_{min}}{g} < \alpha \leqslant \frac{x_{max}}{1 - g}$

\begin{enumerate}[a.]
\item $x_{max} \leqslant x^\sim$

Here we have to consider two cases:
$$
\varphi(x_{min}) -\frac{x_{min}}{\alpha} \geqslant \varphi(x^\sim) - \frac{x^\sim}{\alpha}.
$$

and
$$
\varphi(x_{min}) -\frac{x_{min}}{\alpha} < \varphi(x^\sim) - \frac{x^\sim}{\alpha}
$$

Let us begin with the first one.

\begin{figure}[h]\label{III2a}
\centering{$\includegraphics[trim=0 50 0 35, clip, width=0.85\linewidth]{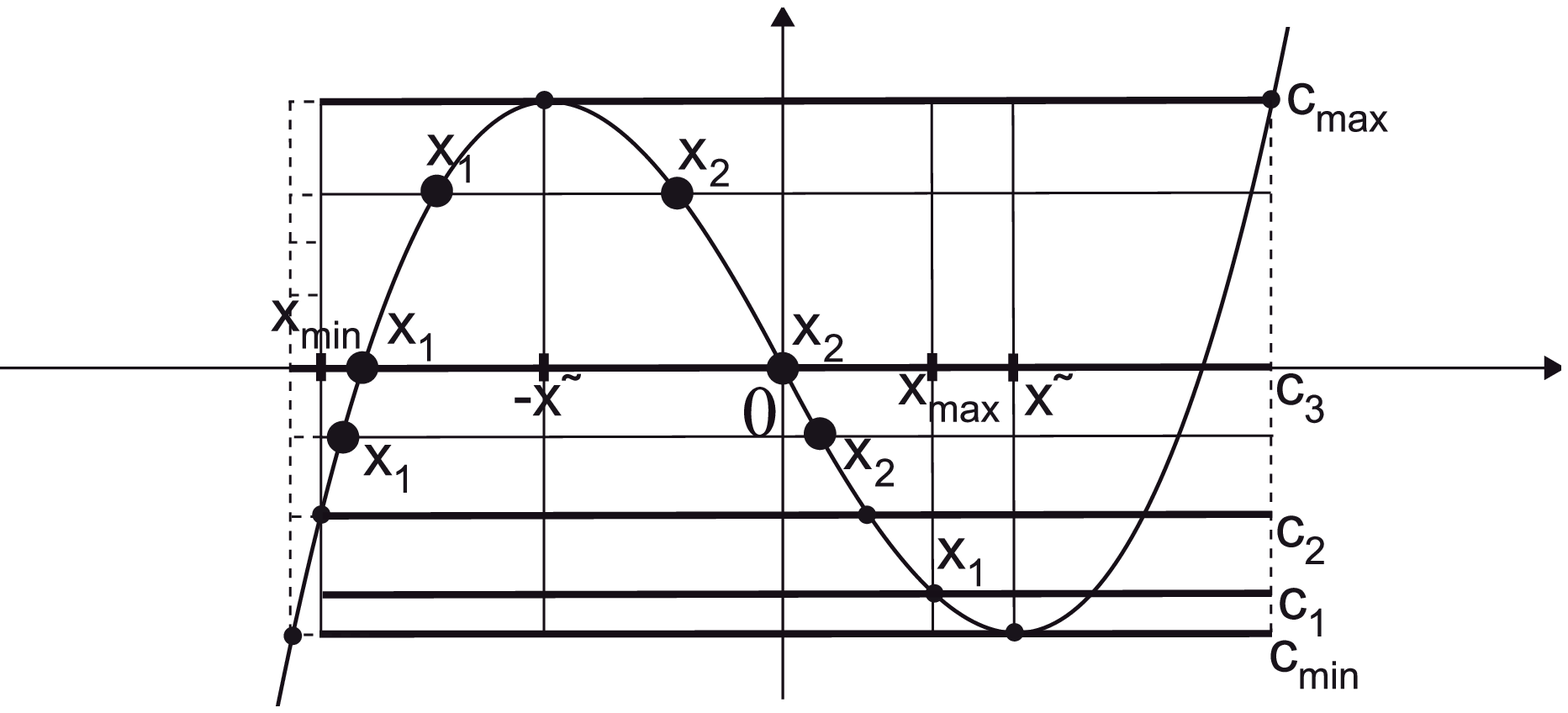}$}
\caption{Case III.2.a, $\varphi(x_{min}) -\frac{x_{min}}{\alpha} \geqslant \varphi(x^\sim) - \frac{x^\sim}{\alpha},$}
\end{figure}

Define $c_1 = \varphi(x_{max}) - \frac{x_{max}}{\alpha},\,$ $c_2 = \varphi(x_{min}) - \frac{x_{min}}{\alpha},$ 
and $c_3 = \varphi(0).$
Thus, in this case the partition in $c$ is as follows:

\begin{enumerate}[1)]
\item $c = c_{min}$

Here $X_C=\left\lbrace x^\sim\right\rbrace,$ which gives us either a trajectory 
of type Ia or a trajectory of type Ib.  \ssk 

\item $c_{min} < c < c_2$

Here $X_C=\left\lbrace x_1\, x_2\right\rbrace,$ where $x_1 \in]0,x^\sim[$ 
and $x_2 \in]x^\sim, x_{max}[,$ which gives us a trajectory of type Ia 
or Ib again (with $x(t_1)=x_1$ or $x(t_1)=x_2$). \ssk 

\item $c_2\leqslant  c<c_3.$

Here $X_C=\left\lbrace x_1, x_2 \right\rbrace,$ where $x_1 \in]x_{min},-x^\sim[,$ 
$x_2 \in ]0, x^\sim[,$ which gives the trajectories of types Ia, Ib, 
IIa, IIb. \ssk 

\item $c=c_3.$

Here $X_C=\left\lbrace x_1, x_2=0\right\rbrace,$
where $x_1\in]x_{min},-x^\sim[,$ which gives us trajectories of types 
Ib, IIb.  \ssk 

\item $c_3 <  c<c_{max}.$

Here $X_C=\left\lbrace x_1, x_2\right\rbrace,$ where $x_1\in]x_{min}, -x^\sim[$ 
and $x_2\in]-x^\sim,0[,$ and due to Lemmas \ref{xProps},  \ref{psiProps}, 
and \ref{xMInopCase}, we have trajectories of types Ia, Ib, IIa, IIb.
\ssk

\item $c=c_{max}$

Here $X_C=\left\lbrace x_1=-x^\sim \right\rbrace$ which does not give us 
new trajectories according to Lemma \ref{xMInopCase}.

\end{enumerate}
\ssk

Now, consider the case $\varphi(x_{min}) -\frac{x_{min}}{\alpha} < 
\varphi(x^\sim) - \frac{x^\sim}{\alpha}.$

\begin{figure}[h]\label{III2a}
\center{$\includegraphics[trim = 0 40 0 30, clip, width=0.8\linewidth]{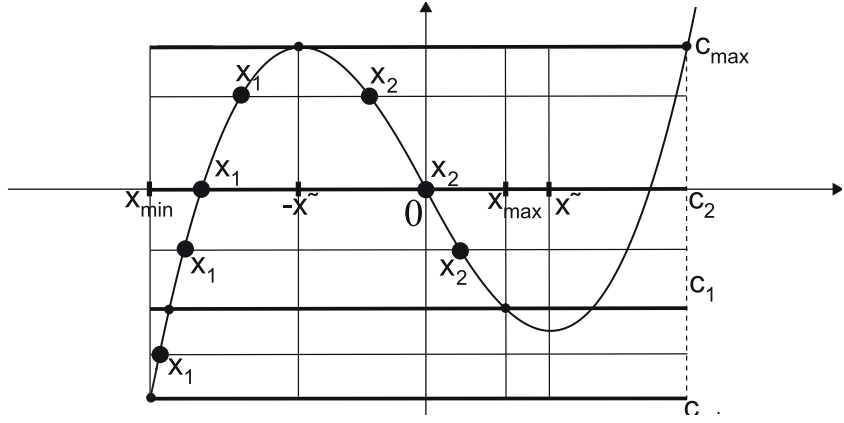}$}
\caption{Case III.2.a, $\varphi(x_{min}) -\frac{x_{min}}{\alpha} < \varphi(x^\sim) - \frac{x^\sim}{\alpha}$}
\end{figure}

The only difference with the preceding case is the following item 
in the partition of $c$: $\; c_{min}<c<c_1\,.$ Here we obtain 
$X_C = \{x_1\},$ where $x_1\in ]x_{min},x^\sim[,$ which does not give us 
new extremals.  \ssk 

\item $x_{max} > x^\sim$

Here we have to consider two cases:
$$
\varphi(x_{min}) -\frac{x_{min}}{\alpha} 
\geqslant \varphi(x^\sim) - \frac{x^\sim}{\alpha}.
$$

and
$$
\varphi(x_{min}) -\frac{x_{min}}{\alpha} < 
\varphi(x^\sim) - \frac{x^\sim}{\alpha}
$$
Let us begin with the first one.

\begin{figure}[h]\label{III2b}
\centering{$\includegraphics[trim = 0 70 0 25, clip, width=0.85\linewidth]{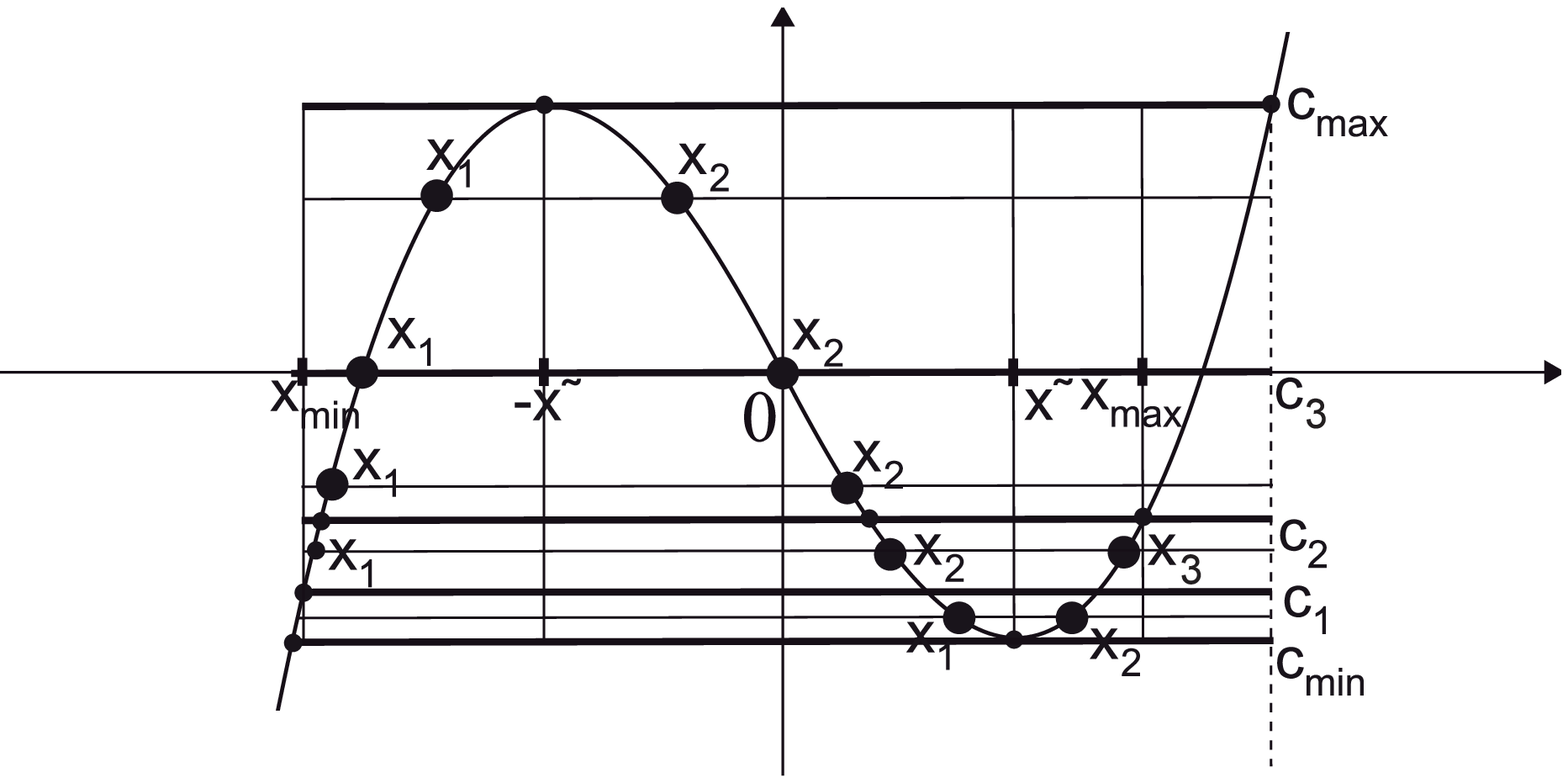}$}
\caption{Case III.2.b, $\varphi(x_{min}) -\frac{x_{min}}{\alpha} \geqslant \varphi(x^\sim) - \frac{x^\sim}{\alpha},$}
\end{figure}

Define $c_1 = \varphi(x_{max}) - \frac{x_{max}}{\alpha},$ 
$\,c_2 = \varphi(x_{min}) - \frac{x_{min}}{\alpha},$ and $c_3 = \varphi(0).$
Thus, in this case the partition of $c$ is as follows:

\begin{enumerate}[1)]
\item $c = c_{min}$

Here $X_C=\left\lbrace x^\sim\right\rbrace,$ which gives us either 
a trajectory of type Ia or a trajectory of type Ib. \ssk  

\item $c_{min} < c < c_1$

Here $X_C=\left\lbrace x_1\, x_2\right\rbrace,$ where $x_1 \in]0,x^\sim[$ 
and $x_2 \in]x^\sim, x_{max}[,$ which gives us a trajectory of type Ia 
or Ib again (with $x(t_1)=x_1$ or $x(t_1)=x_2$). \ssk 

\item $c_1\leqslant  c<c_2.$

Here $X_C=\left\lbrace x_1, x_2, x_3 \right\rbrace,$ where $x_1 \in]x_{min},-x^\sim[,$ 
$x_2 \in ]0, x^\sim[$ and $x_3\in]x^\sim,x_{max}[,$ which gives  
trajectories of types Ia---III. \ssk 

\item $c_2 <  c<c_3.$

Here $X_C=\left\lbrace x_1, x_2\right\rbrace,$ where $x_1 \in]x_{min},-x^\sim[,$ 
$x_2 \in ]0, x^\sim[$ which gives trajectories of types Ia---IIb.
\ssk 

\item $c=c_3.$

Here $X_C=\left\lbrace x_1, x_2=0\right\rbrace,$
where $x_1\in]x_{min},-x^\sim[,$ which gives either a trajectory of 
type 0\, or trajectories of types Ib, IIb.

\item $c_3 <  c<c_{max}.$

Here $X_C=\left\lbrace x_1, x_2\right\rbrace,$ where $x_1\in]x_{min}, -x^\sim[$ 
and $x_2\in]-x^\sim,0[,$ and due to the Lemma \ref{xProps}, \ref{psiProps}, 
and \ref{xMInopCase} we obtain trajectories of types Ia, Ib, IIa, IIb.
\ssk

\item $c=c_{max}$

Here $X_C=\left\lbrace x_1=-x^\sim \right\rbrace$ which gives no new 
trajectories according to Lemma \ref{xMInopCase}.

\end{enumerate}
\ms

Now let us consider the case $\varphi(x_{min}) -\frac{x_{min}}{\alpha} < \varphi(x^\sim) - \frac{x^\sim}{\alpha}.$

\begin{figure}[h]\label{III2a}
\centering{$\includegraphics[trim = 0 30 0 30, clip, width=0.8\linewidth]{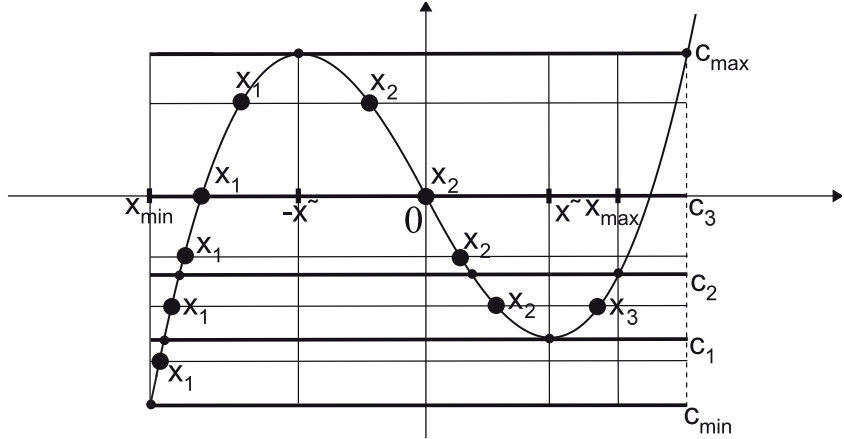}$}
\caption{Case III.2.b, $\varphi(x_{min}) -\frac{x_{min}}{\alpha} < \varphi(x^\sim) - \frac{x^\sim}{\alpha},$}
\end{figure}

The only difference with the preceding case is the following item 
in the partition of $c$: $\; c_{min}<c<c_1$. Here we obtain $X_C = \{x_1\},$ 
where $x_1\in ]x_{min},x^\sim[,$ which does not give new trajectories.
\end{enumerate}  \ms

\item $\alpha > \frac{x_{max}}{1 - g}$

The situation is the same as in the case I.3.

\end{enumerate}

\end{enumerate}

Summing up, we obtain the following possible types of extremals:

\begin{enumerate}[-]

\item Type Ia
 $u=(1,0)$ on consecutive intervals $]0, t_1[,$ $]t_1, T[.$

 \item Type Ib
 $u=(1,u_{sing},0)$ on consecutive intervals $]0,t_1[,$ $]t_1,t_2[,$ $]t_2,T[.$

 \item Type IIa
  $u=(0,1,0)$ on consecutive intervals $]0,t_1[,$ $]t_1,t_2[,$ $]t_2,T[.$
  \item Type IIb
  $u=(0,1,u_{sing},0)$ on consecutive intervals $]0,t_1[,$ $]t_1,t_2[,$ $]t_2,t_3[,$ $]t_3,T[.$

   \item Type III
  $u=(1,0,1,0)$ on consecutive intervals $]0,t_1[,$ $]t_1,t_2[,$ $]t_2,t_3[,$ $]t_3,T[.$
\end{enumerate}

Moreover, we propose the following scheme for searching extremals of possible types:

For given $\varphi(x),$ $g,$ take a parameter $\alpha > 0,$ and first check 
whether the function $\Phi(x)$ has positive or negative roots. If not, 
$\alpha$ should be decreased. If yes, compute the following values:
$$  x_{max}: \quad \varphi(x_{max}) = 1 - g, \qq 
  x_{min}: \quad \varphi(x_{min}) = - g,
$$
and
$$
x^\sim \geqslant 0:  \qquad \varphi'(x^\sim) = 1/\alpha\,.
$$
Then choose extremals for the final analysis according to the following 
scheme:

\begin{enumerate}[I]
  	\item $0 < g < 0.5$ \ms 

  	\begin{enumerate}[1]
  		\item $\alpha\leqslant \frac{x_{max}}{1 - g}$
  			\begin{enumerate}[a]
  				\item $x_{max} \leqslant x^\sim$ --- type Ia
  				\item $x_{max} > x^\sim,$ $x_{min} > - x^\sim$ --- types Ia, Ib
  				\item $x_{max} > x^\sim,$ $x_{min} < - x^\sim$ --- types Ia, Ib, IIa
  			\end{enumerate} \ssk 

  		\item $\frac{x_{max}}{1-g}<\alpha<-\frac{x_{min}}{g}$
  			\begin{enumerate}[a]
  				\item $x_{min} \geqslant - x^\sim$ --- types Ia, Ib
  				\item $x_{min} < - x^\sim$ --- types Ia, Ib, IIa
  			\end{enumerate}  \ssk 
  
  		\item $\alpha > - \frac{x_{min}}{g}$ --- types Ia, Ib, IIa, IIb, III
  	\end{enumerate} \ms 

  	\item $g = 0.5$ \ms  

  	\begin{enumerate}[1]
  		\item $\alpha\leqslant 2 x_{max}$
  			\begin{enumerate}[a]
  				\item $x_{max} < x^\sim$ --- type Ia
  				\item $x_{max} \geqslant x^\sim$ --- types Ia, Ib, IIa
  			\end{enumerate} \ssk
  
  		\item $\alpha > 2 x_{max}$ --- types Ia, Ib, IIa, IIb, III
  	\end{enumerate}  \ms  

  	\item $0.5 < g < 1$  \ms

  	\begin{enumerate}[1]
  		\item $\alpha\leqslant - \frac{x_{min}}{g}$
  			\begin{enumerate}[a]
  				\item $x_{min} \geqslant - x^\sim$ --- type Ia
  				\item $x_{min} < - x^\sim,$ $x_{max} \leqslant x^\sim$ --- type Ia
  				\item $x_{max} > x^\sim$ --- types Ia, Ib
  			\end{enumerate}  \ssk 
  
  		\item $-\frac{x_{min}}{g}<\alpha\leqslant\frac{x_{max}}{1-g}$
  			\begin{enumerate}[a]
  				\item $x_{max} \leqslant x^\sim$ --- types Ia,  IIa
  				\item $x_{max} > x^\sim$ --- types Ia, Ib, IIa, IIb, III
  			\end{enumerate}  \ssk
  
  		\item $\alpha >\frac{x_{max}}{1-g}$ --- types Ia, Ib, IIa, IIb, III
  \end{enumerate}
\end{enumerate}
Note that any trajectory of types IIa, IIb and III includes an interval where $x(t) < 0.$ Thus, if we impose the restriction $x(t) \geqslant 0$ for all $t \in [0, T],$ then the optimal trajectory in problem (\ref{probl}) can be only of types Ia or Ib (the standard well known ones for  the classical Goddard problem). In general, these trajectories can also include an interval where $x(t) < 0,$ but this interval can only end at $T.$

Note also that if $\varphi'(0) \geqslant \frac{1}{\alpha}$ then $\Phi(x)$ has only zero root on $[x_{min}, x_{max}],$ and only trajectories of type Ia can satisfy the MP. 
\begin{rem}
Since always $\psi(t) \geqslant 0,$ the Pontryagin function (\ref{H}) is concave in $(x,u)$ on the convex set $\R \times [0,1].$ Therefore, if we obtain an extremal with $x(t) \geqslant 0$ for all $t \in [0,T],$ then, as is well known (see, e.g., \cite[Theorem 4 in Sect. 2.6]{SeierSid}), it provides the maximum in the problem (\ref{probl}) among all admissible processes with nonnegative velocity component.
\end{rem}

Thus, we obtained a full classification of extremals for problem
(\ref{probl}).

\section{Case of a quadratic resistance function}

\begin{equation}\label{quadrpsi}
\varphi(x)=\begin{cases}
\begin{aligned}
\frac{b}{2} x^2 + k x,& \quad x\geqslant 0,\\
-\frac{b}{2} x^2 + k x,& \quad x < 0,
\end{aligned}
\end{cases}
\end{equation}
where $k>0,\quad b>0.$


In this section we use this special form of $\varphi(x)$ to discard trajectories of type III. Since this type of trajectories can appear only if $\varphi'(0) < \frac{1}{\alpha},$ in further considerations we take $0 < \alpha < \frac{1}{k}.$

\subsection{Additional properties of $\psi(t)$}

Recall the formula for the second time derivarive of $\psi(t):$ (see (\ref{psiSecondDeriv})):

$$
\ddot{\psi} = \dot{\psi}\varphi'(x) + \psi \varphi''(x) \dot{x}.
$$
and prove the following fact.
\begin{lemma}\label{ddotPsiNeg}
Let $\psi(t') = \alpha$ and $\psi(t) < \alpha$ in the right neighbourhood of $t'.$ Let $x(t) > 0$ on $]t',t''[$ and $x(t'')=0.$ Then $\ddot{\psi}(t) < 0$ on $]t', t''[.$
\end{lemma}

\begin{proof}
Obviously, $u = 0$ in the right neighbourhood of $t',$ thus $\dot x < 0$ in the right neighbourhood of $t'$ according to Lemma \ref{xdot}, and $\varphi'(x)$ decreases. In addition, we have $\dot{\psi}(t') = -1 + \alpha \varphi\left(x(t')\right) \leqslant 0.$ Since $\psi(t) < \alpha$ and $\varphi'(x) < \varphi'\left(x(t')\right)$ in the right neighbourhood of $t',$ we get $\dot{\psi}(t) < 0$ in the right neighbourhood of $t',$ thus we get $\psi(t) < \alpha.$ Thus, we get $\psi
(t) < \alpha,$ $u = 0,$ $\dot x < 0,$ $\dot{\psi} < 0$ on $]t',t''[,$ and (\ref{psiSecondDeriv}) implies $\ddot{\psi}(t) < 0$ on $]t', t''[.$ 
\end{proof}

Note that the transversality condition $\psi(T) = 0$ implies $\dot{\psi}(T) = -1$ and $\ddot{\psi}(T) = -\varphi'\left(x(T)\right) < 0.$

The previous results were formulated for the general case of resistance function and for any extremal trajectory. Now we begin to work with the quadratic function and the trajectory of type III.

Remind that for the trajectory of type III we have

$$
X_C = \left\lbrace x_1, x_2, x_3\right\rbrace,
$$ 
where $x_ 1 \in ]x_{min}, - x^\sim[,$ $x_2 \in ]0, x^\sim[,$ $x_3 \in ]x^\sim, x_{max}[$ (see classification in section 5), and 

$$u=(1,0,1,0)$$ 
on consecutive intervals $]0,t_1[,$ $]t_1,t_2[,$ $]t_2,t_3[,$ $]t_3,T[,$ such that $x(t_1) = x_3,$ $x(t_2) = x_1,$ $x(t_3) = x_2,$ and $\psi(t_1) = \psi(t_2) = \psi(t_3) = \alpha.$
Note that for the quadratic function $\varphi(x)$ we get 

\begin{enumerate}[]
\item
$x^\sim = \frac{1 - k\alpha}{b\alpha}$ as the positive root of the equation $\varphi'(x) = \frac{1}{\alpha}$ and

\item
$x_{min} = \frac{k - \sqrt{k^2 + 2 b g}}{b}$ 
as the negative root of the equation $-\varphi(x) - g = 0,$ which reads as $\frac{b}{2} x^2 - k x - g = 0.$
\end{enumerate}
Along the trajectory of type III $x(t)$ may change its sign as described in the following

\textbf{Scheme I}
\begin{enumerate}[1.]
\item $x(t)$ does not change its sign on $]t_1, t_2[.$ Thus, according to Lemma \ref{ddotPsiNeg}, we get $\dot{\psi}(t) < 0$ and $\ddot{\psi}(t) < 0$ on $]t_1, t_2[,$ thus $\psi(t_2) < \psi(t_1),$ which contradicts $\psi(t_1) = \psi(t_2).$

\item $x(t)$ changes its sign on $]t_1, t_2[,$ i.e. there exists $t^{*}\in ]t_1, t_2[$ such that $x(t^{*}) = 0.$ Thus, the following two situations might be realized.
	\begin{enumerate}[2.1]
		\item $x(t)$ does not change its sign on $]t_3, T[.$ 					  Then, according to Lemma \ref{ddotPsiNeg}, we get 				  $\dot{\psi}(t) < 0$ and $\ddot{\psi}(t) < 0$ on 					  $]t_3, T[,$ thus there exists such $T$ that 
			  $\psi(T) = 0.$
		\item $x(t)$ changes its sign on $]t_3, T[,$ i.e. 					  there exists $t^{**}\in ]t_3, T[$ such that $x(t^{**})= 0.$ 
	\end{enumerate}
\end{enumerate}

What is more, the signs of $\ddot{\psi}(t^{*}),$ $\ddot{\psi}(t^{**})$ may be as follows.
\begin{enumerate}[1)]
\item $\ddot{\psi}_{right}(t^{*}) \leqslant 0$ and $\ddot{\psi}_{right}(t^{**}) \leqslant 0,$

\item $\ddot{\psi}_{right}(t^{*}) \leqslant 0$ and $\ddot{\psi}_{right}(t^{**}) > 0,$

\item $\ddot{\psi}_{right}(t^{*}) > 0$ and $\ddot{\psi}_{right}(t^{**}) \leqslant 0,$

\item $\ddot{\psi}_{right}(t^{*}) > 0$ and $\ddot{\psi}_{right}(t^{**}) > 0,$
\end{enumerate}
inducing different behaviour of $\psi(t)$ on $]t^{*}, t_2[$ and $]t^{**}, T[.$ Equation for the second derivative (\ref{psiSecondDeriv}) gives us

\begin{equation}\label{psiDDeriv2}
\ddot{\psi}_{right}(t) = - \varphi(t)' + \psi(t) \left(\varphi'^2\left(x(t)\right) - \varphi_{right}''\left(x(t)\right)\left(\varphi\left(x(t)\right) + g \right)\right).
\end{equation}
The following two facts take place.

\begin{lemma}\label{psiDDerivNeg}
Let $u = 0$ on $]t',t''[$ and $\bar{t} \in ]t', t''[$ be such that $x(t) > 0$ on $]t', \bar{t}[$, $x(\bar{t}) = 0,$ $\dot{\psi}(t) < 0$ in a left neighbourhood of $\bar{t},$ and $\ddot{\psi}_{right}(\bar{t}) \leqslant 0.$ Then $\ddot{\psi}(t) < 0$ on $]\bar{t}, t''[.$
\end{lemma}
\begin{proof}
In the quadratic case, inequality $\ddot{\psi}_{right}(\bar{t}) \leqslant 0$ reads 
$$
\psi(\bar{t}) \leqslant \frac{k}{k^2 + bg}.
$$
Consider the following two cases.

If $\ddot{\psi}_{right}(\bar{t}) < 0,$ i.e. $\psi(\bar{t}) < \frac{k}{k^2 + bg},$ then, according to (\ref{psiDDeriv2}), there exists $\widehat{t}$ such that $\ddot{\psi}(\widehat{t}) = 0,$ i.e.

$$
\psi(\widehat{t}) = \frac{\varphi'(\widehat{x})}{\varphi'^2(\widehat{x}) - \varphi''(\widehat{x}) \left(\varphi(\widehat{x}) + g\right)},\quad \mbox{where}\quad \widehat{x} = x(\widehat{t}),
$$
and $\ddot{\psi}(t) < 0$ on $]\bar{t},\widehat{t}[$ (thus, $\dot{\psi} < 0$ on $]\bar{t},\widehat{t}[$). 
Check that for the quadratic case $\psi(\widehat{t}) > \frac{k}{k^2 + bg},$ i.e. we have a contradiction with $\dot{\psi}(t) < 0$ on $]\bar{t},\widehat{t}[.$
This is equivalent to check that

\begin{equation}\label{ineq1}
\frac{-b \widehat{x} + k}{ (-b \widehat{x} + k)^2 + b (-\frac{b}{2} \widehat{x}^2 + k v{x} + g)} > \frac{k}{k^2 + bg}
\end{equation}
for all $\widehat{x} \in ]x_{min}, 0[,$ which is equivalent to

\begin{equation}\label{quadrEq}
A(x) := k \widehat{x}^2 + 2 g \widehat{x} < 0
\end{equation}
for all $\widehat{x} \in ]x_{min},0[.$

Note that $x(t)\in ]x_{min}, x_{max}[$ for all $t$ according to Lemma \ref{xdot}. Thus, 

$$
x_{min} < \widehat{x} < 0.
$$

Obviously, $A(x) < 0$ for all $\widehat{x} \in \left]-\frac{g}{k},0\right[.$ One can easily check that $x_{min} > - \frac{g}{k},$ so (\ref{quadrEq}) and then (\ref{ineq1}) hold, a contradiction. Thus, $\ddot{\psi}(t) < 0$ on $]\bar{t}, t''[.$ 

If $\ddot{\psi}_{right}(\bar{t}) = 0,$ i.e. $\psi(\bar{t}) = \frac{k}{k^2 + bg},$ we write 

$$
\ddot{\psi}(t) = -z + \frac{1}{2}\psi \left(z^2 + k^2 + 2 bg\right),
$$
where $z = - (b x - k)$ (one can easily check it by substituting $z$ to (\ref{psiDDeriv2})).

Since $\dot \psi(t) < 0$ in the right neighbourhood of $\bar{t},$ we get

$$
\ddot{\psi}(t) \leqslant -z + \frac{1}{2} \frac{k}{k^2 + bg} \left( z^2 + k^2 + 2 b g \right)
$$
in this right neighbourhood of $\hat{t},$ where
$z > k$ and increases  (since $x < 0$ and decreases).
This case may be reduced to the previous one. To do this, compare
$$
-z + \frac{1}{2} \frac{k}{k^2 + bg} \left( z^2 + k^2 + 2 b g \right)\quad\mbox{and}\quad 0,
$$
which is equivalent to comparison of
$$
A_1(z) :=k z^2 - 2 (k^2 + bg) z + k^3 + 2 k b g \quad\mbox{and}\quad 0
$$
for $z > k.$

One can easily find the roots

$$
z_1 = k, \quad z_2 = \frac{k^2 + 2 b g}{k} \quad\mbox{of}\quad A_1(z),
$$
thus $A_1(z) < 0$ in the corresponding right neighbourhood of $k,$ which implies $\ddot{\psi}(t) < 0$ in the right neighbourhood of $\hat{t},$ i.e. we reduced the case to the previous one and can use the corresponding proof.
\end{proof}
\ms
\begin{lemma}\label{psiDDerivPos}
Let $u = 0$ on $]t',t''[$ and $\bar{t} \in ]t',t''[$ be such that $x(t) > 0$ on $]t',\hat{t}[,$ $x(\bar{t}) = 0,$ $\dot{\psi}(t) < 0$ in a left neighbourhood of $\bar{t},$ and $\ddot{\psi}_{right}(\bar{t}) > 0.$ Then $\ddot{\psi}(t) > 0$ on $]\bar{t}, t''[.$
\end{lemma}
\begin{proof}
In the quadratic case inequality $\ddot{\psi}_{right}(\bar{t}) > 0$ reads 
$$
\psi(\bar{t}) > \frac{k}{k^2 + bg},
$$
and, according to (\ref{psiDDeriv2}), $\ddot{\psi}(t) > 0$ on $]\bar{t},\widehat{t}[,$ where $\ddot{\psi}(\widehat{t}) = 0,$ i.e.

$$
\psi(\widehat{t}) = \frac{\varphi'(\widehat{x})}{\varphi'^2(\widehat{x}) - \varphi''(\widehat{x}) \left(\varphi(\widehat{x}) + g\right)},\quad \mbox{where}\quad \widehat{x} = x(\widehat{t}).
$$

Let $t'''$ be such that $\dot{\psi}(t''') = 0.$ If $\widehat{t} > t''',$ then (\ref{psiDDeriv2}) gives $\ddot{\psi}(t) > 0$ and increases, a contradiction with $\ddot{\psi}(\widehat{t}) = 0.$ Thus, $\widehat{t} \leqslant t''',$ i.e. $\dot{\psi}(t) < 0$ on $]\bar{t},\widehat{t}[.$ and $\psi(\widehat{t}) < \psi(\bar{t})$.

Analogically to Lemma \ref{psiDDerivNeg} we check that $\psi(\widehat{t}) > \frac{k}{k^2 + bg},$ i.e. we have a contradiction with $\dot{\psi}(t) < 0$ on $]\bar{t},\widehat{t}[.$ Thus, $\ddot{\psi}(t) > 0$ on $]\bar{t}, t''[.$ 
\end{proof}

Show that for the trajectory of type III only cases from the following scheme are realized.

\textbf{Scheme II}
\begin{enumerate}[1.]
\item $x(t)$ does not change its sign on $]t_1,t_2[,$
\item $x(t)$ changes its sign on $]t_1,t_2[$ and $\ddot{\psi}(t^{*})\leqslant 0,$ i.e.
\begin{equation}\label{psitreshold1}
\psi(t^{*}) \leqslant \frac{k}{k^2 + bg},
\end{equation}
\item $x(t)$ changes its sign on both of $]t_1,t_2[$ and $]t_3,T[,$ and $\ddot{\psi}(t^{*}) > 0,$ $\ddot{\psi}(t^{**}) > 0,$
\begin{equation}\label{psitreshold2}
\psi(t^{*}) > \frac{k}{k^2 + bg}, \quad \psi(t^{**}) > \frac{k}{k^2 + bg}.
\end{equation}
\end{enumerate}
Then case 1 leads to a contradiction with $\psi(t_2) = \psi(t_1)$ according to Lemma \ref{ddotPsiNeg}, case 2 leads to a contradiction with $\psi(t_2) = \psi(t_1)$ according to Lemma \ref{psiDDerivNeg} and case 3 leads to a contradiction with $\ddot{\psi}(T) < 0$ according to Lemma \ref{psiDDerivPos}.

Consider intervals $]t_1, t^{*}[$ and $]t_3,t^{**}[.$
Since $x(t) > 0$ on $]t_1, t^{*}[ \cup ]t_3,t^{**}[,$ we get

$$
\dot \psi(t) = - 1 + \psi(t) \left(b x(t) + k\right),
$$

$$
\dot x(t) = - \frac{b}{2}x^2(t) - k x(t) - g
$$
on $]t_1, t^{*}[ \cup ]t_3,t^{**}.[$
Thus,

$$
\frac{d \psi}{d x} = \frac{1 - \psi(x) (bx + k)}{\frac{b}{2}x^2 + k x + g},
$$
and, denoting $b x + k = z,$ we get

\begin{equation}\label{psiQuadrEq}
\begin{cases}
\begin{aligned}
\frac{d \psi}{d z} &= 2 \frac{1 - \psi(z) z}{2 b g - k^2 + z^2},\\
\psi(z_0) &= \alpha,
\end{aligned}
\end{cases}
\end{equation}
where $z_0 = b x_0 + k,$ $x_0$ is either $x(t_1)$ or $x(t_3).$

Writing (\ref{psiQuadrEq}) in form

$$\frac{d \psi}{d z} + \frac{2 z}{2 b g - k^2 + z^2} \psi = \frac{2}{2 b g - k^2 + z^2},
$$
it is easy to obtain
\begin{equation}\label{psiEquationQuadr}
\psi(z) = \frac{\alpha(2bg - k^2 + z_0^2) + 2(z - z_1)}{2bg - k^2 + z^2},
\end{equation}
In the next subsection we use formula (\ref{psiEquationQuadr}) to prove the inoptimality of the trajectory of type III.
\subsection{Inoptimality of trajectory of type III}
Following \textbf{Scheme II}, we substitute $z(0) = k$ to (\ref{psiEquationQuadr}), and compare

\begin{equation}\label{psiKIneq}
\psi(k) = \frac{\alpha (2 b g - k^2 + z_0^2) + 2 (k - z_0)}{2 b g}\quad\mbox{and}\quad0.
\end{equation}
Consider the following two cases. 

\begin{enumerate}
\item $2bg \geqslant k ^2$ 

Since $z(0) = k,$ we denote $2 b g - k^2 := h^2$ and rewrite (\ref{psiKIneq}) as comparison of

$$
\frac{\alpha (h^2 + z_0^2) + 2 (k - z_0)}{2 b g}\quad\mbox{and}\quad0,
$$

which is equivalent to a comparison of

$$
B(z_0) := \alpha z_0^2 - 2 z_0 + 2 k + \alpha h^2\quad\mbox{and}\quad0
$$

for $z_0$ taken equal to $z_2 := b x_2 + k$ and $z_3 := b x_3 + k.$

Discriminant of quadratic function $B(z_0)$ is equal to

$$
D_B(\alpha) =  4 \left( 1 - h^2 \alpha^2 - 2 k \alpha \right)
$$

with the roots

$$
\alpha_1 = \frac{-k - \sqrt{k^2 + h^2}}{h^2}, \qquad \alpha_2 = \frac{-k + \sqrt{k^2 + h^2}}{h^2}.
$$

One can easily check that $0 < \alpha_2 < \frac{1}{k},$ thus we have either $D_B(\alpha) < 0$ for all considered $\alpha \in ]0, 1/k[$ (thus, $B(z_0) > 0$ for all $z_0$ and $x(t)$ changes its sign on both of $]t_1, t_2[$ and $]t_3,T[$), or $B(z_0)$ has two roots

$$
z_{min} = \frac{1 - \sqrt{1 - h^2 \alpha^2 - 2 k \alpha}}{\alpha}, \qquad z_{max} = \frac{1 + \sqrt{1 - h^2 \alpha^2 - 2 k \alpha}}{\alpha}.  
$$

We do not consider the case $D_B(\alpha) = 0,$ since in this case $B(z_0) > 0$ for all $z_0 \neq z^\sim := \frac{1}{\alpha},$ and the root $z^\sim$ corresponds to the singular arc (see classification in section 5), while we work with a boundary arcs$]t_1, t_2[$ and $]t_3,T[,$ where $u = 0.$

\begin{figure}[h]\label{ASigns}
\begin{minipage}[h]{0.48\linewidth}
\center{$\includegraphics[width=\linewidth]{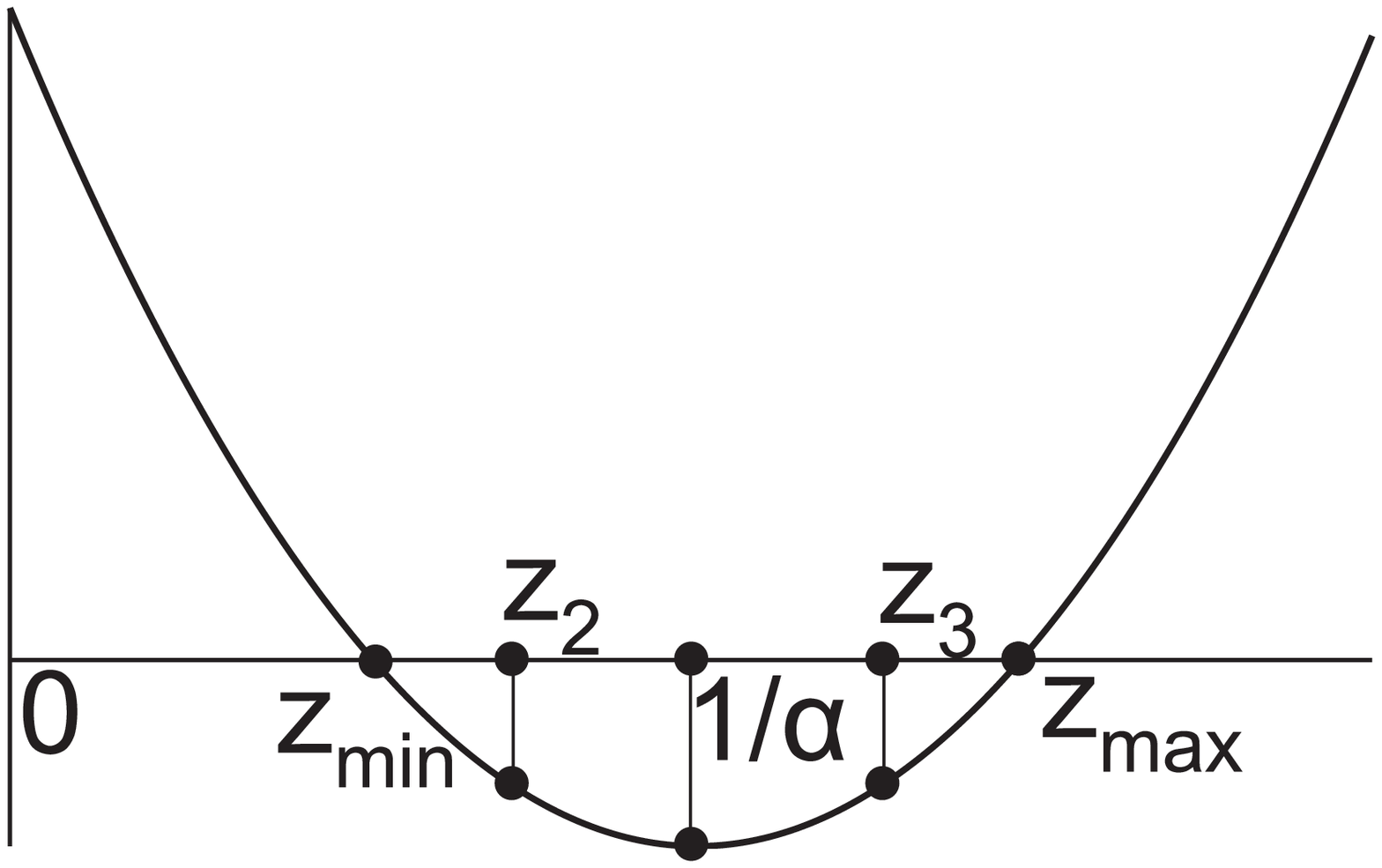}$
\center{a: $B(z_2) \leqslant 0,$ $B(z_3) \leqslant 0$}
}
\end{minipage}
\hfill
\begin{minipage}[h]{0.48\linewidth}
\center{$\includegraphics[width=\linewidth]{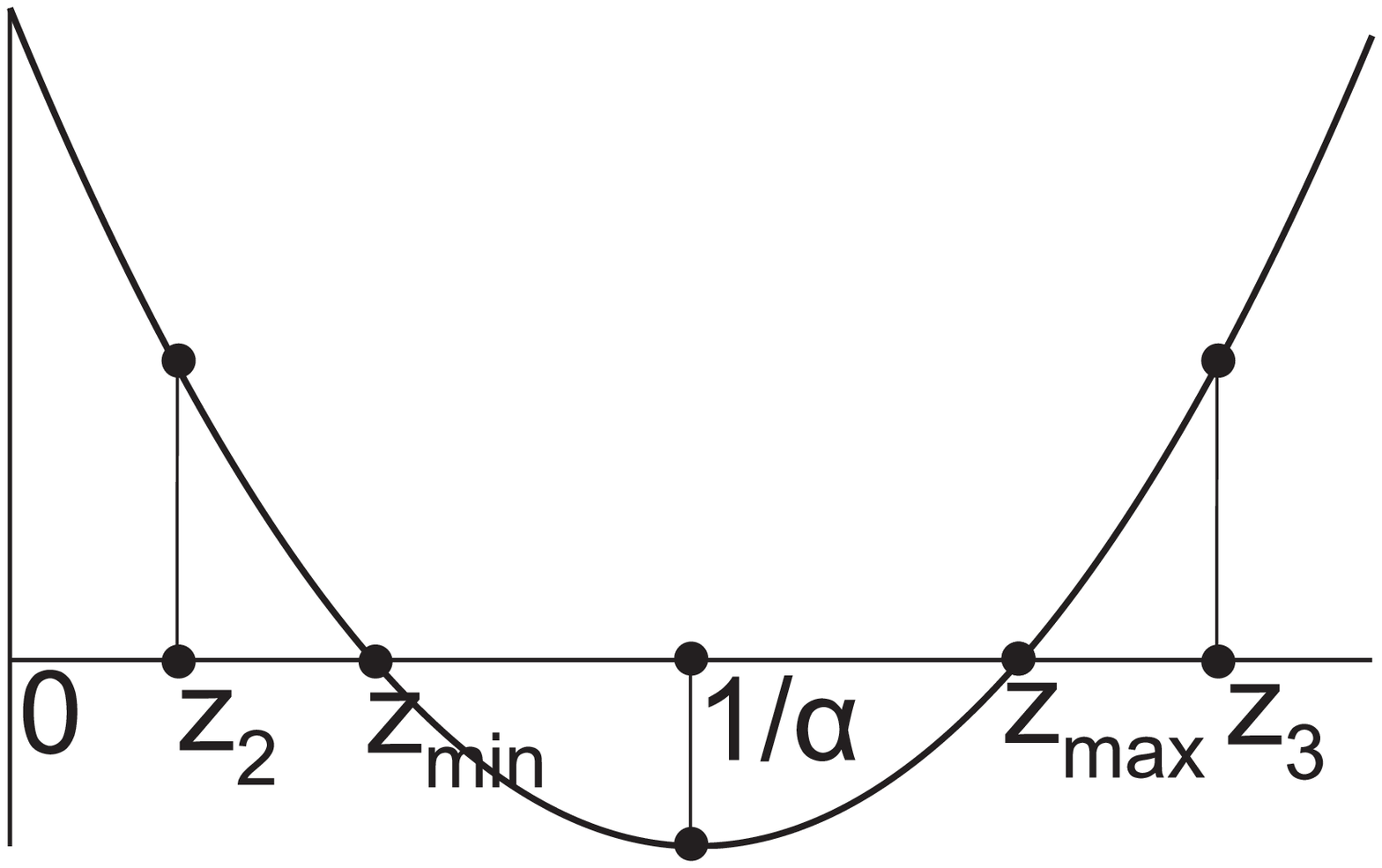}$}
\center{b: $B(z_2) > 0,$ $B(z_3) > 0$}
\end{minipage}
\caption{Graph of $B(z_0)$ and $C(z_0)$ (both of them are quadratic functions)}
\end{figure}

Since $x_2, x_3$ are two positive roots of the equation

$$
\Phi(x) - c :=\frac{b}{2} x^2 + x\frac{k\alpha - 1}{\alpha} - c = 0,
$$

and they are symmetric w.r.t. $x^\sim,$ the corresponding $z_2 = b x_2 + k$ and $z_3 = b x_3 + k$ are symmetric w.r.t. $z^\sim = b x^\sim + k =  \frac{1}{
\alpha},$ which means that $B(z_1)$ and $B(z_2)$ have the same sign (see Fig. 15) and for every collection of parameters such that $2 b g \geqslant k^2$ we get either the case 1 of the \textbf{Scheme II} (thus, trajectory of type III does not satisfy the MP), or the case where we have to compare

$$
\psi(k) := \frac{\alpha(2bg - k^2 + z_0^2) + 2 (k - z_0)}{2bg}\quad\mbox{and}\quad\frac{k}{k^2 + bg}.
$$

for $z_0$ taken equal to $z_2 := b x_2 + k$ and $z_3 := b x_3 + k.$

Denoting $b g + k^2 := c^2,$ this is equivalent to comparison of

\begin{equation}\label{eq}
C(z_0):= c^2 \alpha z_0^2 - 2 c^2 z_0 + \alpha h^2 c^2 + 2 k^3 \quad\mbox{and}\quad 0.
\end{equation}

Discriminant of the quadratic function $C(z_0)$ is equal to 

$$
D_C(\alpha)= 4 \left(- h^2 c^4 \alpha^2 - 2 k^3 c^2 \alpha + c^4 \right)
$$

with the roots

$$
\alpha_1 = \frac{-k^3 - \sqrt{k^6 + h^2 c^4}}{h^2 c^2},\quad
\alpha_2 = \frac{-k^3 + \sqrt{k^6 + h^2 c^4}}{h^2 c^2}.
$$

One can easily check that

$$
\alpha_1 < 0 < \alpha_2 < \frac{1}{k} ,
$$

thus either $D_C(\alpha) < 0$ and $C(z_0) > 0$ for all $z_0,$ or $C(z_0)$ has two roots

\begin{equation*}
z_{min} = \frac{1 - \sqrt{-h^2 \alpha^2 - 2 \frac{k^3}{c^2} \alpha + 1}}{\alpha},\quad
z_{max} = \frac{1 + \sqrt{-h^2 \alpha^2 - 2 \frac{k^3}{c^2} \alpha + 1}}{\alpha}
\end{equation*}

symmetric to $z^\sim = b x^\sim + k.$ The first situation induces the case 3 of \textbf{Scheme II}, thus trajectory of type III does not satisfy the maximum principle. In the second case, since $z_2 = b x_2 + k$ and $z_3 = b x_3 + k,$ we get $C(z_2)$ and $C(z_3)$ have the same sign (see Fig. 15). Thus, either the case 2 or the case 3 of    \textbf{Scheme II} is realized, thus the trajectory of type III does not satisfy the MP. As before, since the arcs $]t_1, t_2[$ and $]t_3,T[$ are not singular, we do not consider the case $D_C(\alpha) = 0.$

\ms
\item $2bg < k ^2$ 

Analogically to the previous case, since $z(0) = k,$ we denote $2 b g - k^2 := - h^2$ and сompare (see (\ref{psiKIneq}))

$$
\frac{\alpha (- h^2 + z_0^2) + 2 (k - z_0)}{2 b g}\quad\mbox{and}\quad0,
$$

which is equal to the comparison of

$$
B(z_0) := \alpha z_0^2 - 2 z_0 + 2 k - \alpha h^2\quad\mbox{and}\quad0
$$

for $z_0$ taken equal to $z_2 := b x_2 + k$ and $z_3 = b x_3 + k.$

Discriminant of the quadratic function $B(z_0)$ is equal to

$$
D_B(\alpha)= 4 \left(1 + h^2 \alpha^2 - 2 k \alpha\right)
$$

with the roots

$$
\alpha_1 = \frac{k - \sqrt{k^2 - h^2}}{h^2}, \quad \alpha_2 = \frac{k + \sqrt{k^2 - h^2}}{h^2}.
$$

One can easily check that $0 < \alpha_1 < \frac{1}{k} < \alpha_2,$ thus either $D_B(z_0) < 0$ (thus, $B(z_0) > 0$ for all $z_0$ and $x(t)$ changes its sign on both of $]t_1,t_2[$ and $]t_3,T[$), or $B(z_0)$ has two roots

$$
z_{min} = \frac{1 - \sqrt{1 + h^2 \alpha^2 - 2 k \alpha}}{\alpha}, \quad z_{max} = \frac{1 + \sqrt{1 + h^2 \alpha^2 - 2 k \alpha}}{\alpha}.  
$$
As before, since the arcs $]t_1, t_2[$ and $]t_3,T[$ are not singular, we do not consider the case $D_B(\alpha) = 0.$

Since $x_2, x_3$ are two positive roots of the equation

$$
\Phi(x) - c :=\frac{b}{2} x^2 + x\frac{k\alpha - 1}{\alpha} - c = 0,
$$

and they are symmetric w.r.t. $x^\sim,$ the corresponding $z_2 = b x_2 + k$ and $z_3 = b x_3 + k$ are symmetric w.r.t. $z^\sim = b x^\sim + k =  \frac{1}{
\alpha},$ which means that $B(z_1)$ and $B(z_2)$ have the same sign (see Fig. 16) and for every collection of parameters such that $2 b g < k^2$ we get either the case 1 of \textbf{Scheme II} (thus, trajectory of type III does not satisfy the MP), or the case where we have to compare
$$
\frac{\alpha(2bg - k^2 + z_0^2) + 2(k - z_0)}{2bg} \quad\mbox{and}\quad \frac{k}{k^2 + bg}.
$$
for $z_0$ taken equal to $z_2 := b x_2 + k$ and $z_3 = b x_3 + k.$

Denoting $2 b g - k^2 := - h^2,$ $h \neq 0$ and $b g + k^2 := c^2,$ we obtain the equivalent comparison of

\begin{equation}\label{eq1}
C(z_0) := c^2 \alpha z_0^2 - 2 c^2 z_0 - \alpha h^2 c^2 + 2 k^3\quad\mbox{and}\quad 0.
\end{equation}

Discriminant of the quadratic function $C(z_0)$ is equal to

$$
D_C(\alpha) = 4 \left(h^2 c^4 \alpha^2 - 2 k^3 c^2 \alpha + c^4\right)
$$

with the roots

$$
\alpha_1 = \frac{k^3 - \sqrt{k^6 - h^2 c^4}}{h^2 c^2},\quad
\alpha_2 = \frac{k^3 + \sqrt{k^6 - h^2 c^4}}{h^2 c^2}.
$$

One can easily check that

$$
0 < \alpha_1 < \frac{1}{k} < \alpha_2,
$$

thus either $D_C(\alpha) < 0$ and $C(z_0) > 0$ for all $z_0,$ or $C(z_0)$ has two roots

\begin{equation*}
z_{min} = \frac{1 - \sqrt{-h^2 \alpha^2 - 2 \frac{k^3}{c^2} \alpha + 1}}{\alpha},\quad
z_{max} = \frac{1 + \sqrt{-h^2 \alpha^2 - 2 \frac{k^3}{c^2} \alpha + 1}}{\alpha}
\end{equation*}

symmetric to $z^\sim = b x^\sim + k.$ The first situation induces the case 3 of \textbf{Scheme II}, thus the trajectory of type III does not satisfy the maximum principle. In the second case, since $z_2 = b x_2 + k$ and $z_3 = b x_3 + k,$ we get $C(z_2)$ and $C(z_3)$ have the same sign (see Fig. 15). Thus, either the case 2 or the case 3 of the \textbf{Scheme II} is realized, thus the trajectory of type III does not satisfy the MP.

As before, since arcs $]t_1, t_2[$ and $]t_3,T[$ are not singular, we do not consider the case $D(\alpha) = 0.$

\end{enumerate}
Summing up, we obtain the following result.

\begin{theorem}\label{Traj_III_inopt}
If the  media resistance function has the quadratic form (\ref{quadrpsi}), then the trajectories of type III do not satisfy the MP.
\end{theorem}

Thus, for the quadratic function $\varphi(x)$ we have to investigate only trajectories of types Ia, Ib, IIa, IIb.

\section{Conclusion}

We considered a version of the Goddard problem with a flat and constant gravity
field and obtained all types of Pontryagin extremals. Trajectory of every
type contains no more than four arcs (i.e. there are no more than three switching points) corresponding to the control values $u = 0,$
$u = 1$ or singular control $u = u_{sing}\,.$ Moreover, $x(t) = const$
along singular arc, i.e. the flight height changes in linear way.
For every type of trajectory we obtained formulas for switching times and the value of singular control. If $x(t) \geqslant 0$ along the whole extremal, then the optimal trajectory is either of type Ia or of type Ib and delivers a global maximum among all admissible processes with nonnegative velocity component. In the case of quadratic resistance function $\varphi(x),$ we proved that the trajectory of type III (with four boundary arcs) does not satisfy the MP. In the general case, the question about optimality of extremals, as well as the question about their uniqueness is still open. The question about the dependence of the optimal trajectory on the parameters
of the problem is also very interesting. These questions can be
subjects of further investigations of problem (\ref{probl}) and its modifications.

\end{document}